\documentclass[12pt,twoside]{article} % leqno
\usepackage[utf8]{inputenc}
\usepackage{graphicx}
\usepackage{enumitem}
\usepackage{multirow}
\usepackage{latexsym}
\usepackage{amssymb}
\usepackage{mathrsfs}
\usepackage{epsfig}
\usepackage{cite}
\usepackage[dvipsnames]{xcolor}
\usepackage[T1]{fontenc}
\usepackage{amsmath}
\usepackage{float}
\allowdisplaybreaks %erlaubt Seitenumbr\"{u}che in Arrays

\numberwithin{equation}{section}

\def \E{\mbox{E}}

\newtheorem{theo}{Theorem}[section]
\newtheorem{lemma}[theo]{Lemma}
\newtheorem{cor}[theo]{Corollary}

\newtheorem{prop}[theo]{Proposition}

\newtheorem{Assumption}{Assumption}

\def\beq{\begin{eqnarray*}}
\def\eeq{\end{eqnarray*}}

\newcommand{\Dkonv}{\stackrel{\mathcal{D}}{\longrightarrow}}%Pfeil für 
\newcommand{\Pkonv}{\stackrel{P}{\longrightarrow}}%Pfeil für stochastische Konvergenz

\newcommand{\e}{\text{E}}

\newcommand{\nn}{\nonumber}
\newcommand{\ov}{\overline}

\newcommand{\Iwkd}{I\{ \hat w_k \geq \alpha \}}

\newenvironment{proof}{\noindent\sc Proof.\quad\rm}{
    \hspace*{\fill} $\Box$ \vspace{2ex} }

\begin{document}

%\begin{frontmatter}
\title{Testing exogeneity in the functional linear regression model}

\author{Manuela Dorn \\ 
{\small \sl Department of Mathematics, Physics, and Computer Science, University of Bayreuth}\\{\small \sl manuela.dorn@uni-bayreuth.de }\and Melanie Birke\\
{\small \sl Department of Mathematics, Physics, and Computer Science, University of Bayreuth}\\{\small \sl melanie.birke@uni-bayreuth.de} \and Carsten Jentsch\\{\small \sl Department of Statistics, TU Dortmund University}\\{\small \sl jentsch@statistik.tu-dortmund.de}
}

\maketitle

\begin{abstract}
We propose a novel test statistic for testing exogeneity in the functional linear regression model. In contrast to Hausman-type tests in finite dimensional linear regression setups, a direct extension to the functional linear regression model is not possible. Instead, we propose a test statistic based on the sum of the squared difference of projections of the two estimators  for testing the null hypothesis of exogeneity in the functional linear regression model. We derive asymptotic normality under the null and consistency under general alternatives. Moreover, we prove bootstrap consistency results for residual-based bootstraps. In simulations, we investigate the finite sample performance of the proposed testing approach and illustrate the superiority of bootstrap-based approaches. In particular, the bootstrap approaches turn out to be much more robust with respect to the choice of the regularization parameter.
\end{abstract}


Keywords and Phrases:  {asymptotic theory, bootstrap infe
%\end{frontmatter}

%%%%%%%%%%%%%%%%%%%%%%%%%%%%%%%%%%%%%%%%%%%%%%
%%%% Main text entry area:
\section{Introduction}

In functional linear regression models, goodness-of-fit tests are much more complicated to construct then e.g.~in the multiple linear setting. This, among others stems to the fact that in functional linear regression models the $L_2$-distance of the slope function estimator to the true function has no proper limiting distribution. This was shown in \cite{Car06} and \cite{Ruy00} for two estimators in the classical functional linear regression model under exogeneity. It turns out, that the lack of a proper limiting distribution also applies for other estimators using different model assumptions. This phenomenon inherent to functional data setups is probably one of the reasons why goodness-of-fit testing is generally not that widely developed for functional regression models yet. In particular, desirable counterparts of standard tests that are well-established in the multiple linear model are still missing in the functional linear setting.

In functional data settings, existing goodness-of-fit tests are described in \cite{MSM05}, who use a suitable scalar product to transform the functions to a different space using the autocovariance operator to obtain a test statistic having a proper limiting distribution. Further approaches are given in \cite{Cues18} and \cite{Gar14,Gar20}, 
%\footnote{MD: Sollte auch hier zu finden sein: Journal of Computational and Graphical Statistics Vol. 23, 761--778. Habe leider keinen Zugriff darauf.} 
%and \cite{Gar20} 
who use random projections together with empirical process techniques.

In practice, one important model assumption is the exogeneity of the regressor. Especially in economics, this assumption is often violated such that the regressors are correlated with the error terms which leads to endogeneity issues. Estimating in such a model is an inverse problem. Neglecting endogeneity generally results in inconsistent estimators. Hence, it is important to test the data for exogeneity first. If the null hypothesis of exogeneity is rejected, different estimators such as e.g.~instrumental variable (IV) estimators are required to achieve consistent estimation. See \cite{Joh16}, \cite{Flo11} or \cite{FV14}, who consider such IV estimators in functional regression setups and derive asymptotic theory. While in the multiple linear regression model the Hausman test (see \cite{Hau78} und \cite{Wu73})
%\tb{(...)}\footnote{\cite{Hau78} und \cite{Wu73}?
 %Literaturverz.einträge können entfernt werden, falls sie nicht zitiert werden sollen.} 
is a standard and natural approach for testing exogeneity, this method cannot be transferred directly to the functional linear model since it is based on the $L_2$-distance of two slope function estimators due to the following proposition which transfers the results in \cite{Car06} and \cite{Ruy00} to the present setting.

In the following, let $\hat\beta$ denote the estimator
 of the slope function in the exogeneous model described in \cite{Joh13} and \cite{Joh13add}, which is consistent under exogeneity, but inconsistent under endogeneity, and by $\hat\beta_{IV}$
 the IV estimator in the endogeneous functional linear model given in \cite{Joh16}, which is consistent in both cases. Then, we get the following result.

\begin{prop}
\label{noasyd}
In the functional linear regression model \eqref{eq:model} defined below, even under exogeneity, there is no random variable $Z$ with non-degenerate distribution, such that
\begin{equation*}
s_n \| \hat\beta_{IV} - \hat \beta \| \stackrel{\mathcal D}{\to} Z
\end{equation*}
for some real sequence $(s_n)_{n\in \mathbb N}$ with $lim_{n \to \infty} s_n = \infty$, where $||\cdot||$ denotes the norm of the Hilbert space.
\end{prop}
The proof of this result mainly goes along the lines of the one in \cite{Car06}, see also \cite{Dor21} for further details. This is why we just state the result here and use it as motivation for a different approach in the following. Motivated by the fact, that in contrast to the $L_2$-distance, the projection error typically has an asymptotic distribution (see e.g.~\cite{MSM05} and \cite{FV14}),
%\footnote{MD: Auch
%\cite{MSM05} zitieren?}, 
we propose to use the sum of the squared difference of
projections of the two estimators as test statistic.

The rest of the paper is organized as follows. In Section \ref{se:modtest}, we state the model assumptions, construct the test statistic and derive its asymptotic distribution. {As the limiting distribution turns out to depend on unknown functional nuisance parameters, which are difficult to estimate, we propose residual-based bootstrap methods in Section \ref{se:boot} and prove their consistency.} The finite sample performance of all discussed tests is  {investigated} in Section \ref{se:finite}. 
All longer proofs are deferred to the Appendix and additional auxiliary results to a supplement.

\section{Model and test statistic}
\label{se:modtest}
We consider the  {functional linear regression} model
\begin{equation}Y=\int_{[0,1]}X(t)\beta(t)dt+ U=\left\langle \beta,X\right\rangle+ U,\label{eq:model}
\end{equation}
where $Y$ is a real-valued random variable, $U$ is a real-valued error term with $E[U]=0$ and $E[U^2]=\sigma^2 {\in(0,\infty)}$, $X$ is a functional random variable with values in $L_2([0,1])$ such that $\int_0^1 E|X(t)|^2 \,dt < \infty $.  In this setup, the error variance $\sigma^2$ is unknown, and $\beta$ is an unknown function from the Sobolev space of periodically extendable square integrable functions denoted by
\begin{equation} 
\mathcal W_\nu := \Big\{ f \in L^2[0,1] :\; \|f
\|_\nu^2 := \sum_{k \in \mathbb Z} \gamma_k^\nu |\langle f,
\phi_k\rangle|^2 < \infty \Big\},
\label{Sobolevraum}
\end{equation}
where $(\phi_k)_{k\in\mathbb Z}$ is the Fourier basis of $L^2([0,1])$, $\nu\in\mathbb R$ and
$\gamma_k=1+|2\pi k|$, $k\in\mathbb Z$, see e.g. \cite{Neuba,Neubb}, \cite{MR96} or \cite{Tsyb04}. {In the setup of
\eqref{eq:model},} we will speak of \emph{exogeneity}  {(and call $X$ an exogenous regressor)}, if
\begin{equation}
\label{eq:exo}
H_0:\ \ E[X(t)U]=0\mbox{ for all }t\in[0,1].
\end{equation}
 {Otherwise, we will speak of \emph{endogeneity} (and call $X$ an endogeneous regressor), if}
\begin{equation}
\label{eq:endo}
H_1:\ \ E[X(t)U]\neq0\mbox{ for at least one }t\in[0,1].
\end{equation}
For consistent estimation in the endogeneous case, we assume to additionally have a functional instrumental variable $W$ with values in $L_2([0,1])$ such that $\int_0^1 E|W(t)|^2 \,dt < \infty $ and $E[UW(t)]=0$ for all $t\in[0,1]$. For the sake of simplicity, it is often assumed in the literature, that $E[X(t)]=E[W(t)]=0$ holds for all $t\in[0,1]$. However, the general case can be handled along the same lines by centering with the sample mean in a first step and our results are stated for the general case. For estimating the cross-covariance operator, we also assume that $(X,W)$ is second-order stationary, see \cite{Joh16}.

\begin{Assumption}
\label{ass:secstat}
There exist functions $c_X,c_W,c_{WX}:\mathbb [-1,1]\to\mathbb R$, such that
\begin{align*}
Cov(X(s),X(t)) &= c_{X}(t-s),	\\	% \quad s,t \in [0,1],\\
Cov(W(s),W(t)) &= c_{W}(t-s),	\\	%  \quad s,t \in [0,1],\\
Cov(W(s),X(t)) &= c_{WX}(t-s),	%  \quad s,t \in [0,1], 
\end{align*}
for all $s,t \in [0,1]$, respectively, where $c_X$ is assumed to be continuous.
\end{Assumption}

By imposing continuity of $c_X$, we have that whenever \eqref{eq:endo} holds for one $t\in[0,1]$, this immediately implies $E[X(t)U]\neq0$ on some set with positive Lebesgue measure. This condition ensures, that the test statistic proposed in the following can be used to consistently test for the null hypothesis \eqref{eq:exo} against alternatives \eqref{eq:endo}.

Note, that $c_X$ and $c_W$ define the kernels of the covariance operators $\Gamma_X$ of $X$ and $\Gamma_W$ of $W$, respectively, and $c_{WX}$ is the kernel of the cross covariance operator $\Gamma_{WX}$ of $X$ and $W$. The (joint) weak stationarity of $(X,W)$ ensures, that both covariance operators as well as the cross covariance operator have the same exponential system of eigenfunctions, which we denote by $(\phi_k)_{k\in\mathbb N}$. Hence, let $(x_k,\phi_k)_{k\in\mathbb N}$ be the eigensystem of $\Gamma_X$, $(w_k,\phi_k)_{k\in\mathbb N}$ the eigensystem of $\Gamma_{W}$ and $(c_k,\phi_k)_{k\in\mathbb N}$ the eigensystem of $\Gamma_{WX}$. Furthermore, denote $\lambda_k=\frac{|c_k|^2}{w_k}$.

\begin{Assumption}
\label{ass:eigenvalues}
Throughout the article, we assume that all eigenvalues are strictly positive and that
\[\sum_{k\in\mathbb Z}\frac{|E[Y\langle X,\phi_k\rangle]|^2}{x_k^2}<\infty.\]
Furthermore, we denote by $\mu_X=\sum_{k\in\mathbb Z}\langle\mu_X,\phi_k\rangle\phi_k$ and $\mu_W=\sum_{k\in\mathbb Z}\langle\mu_W,\phi_k\rangle\phi_k$ the expectations of $X$ and $W$, respectively. Additionally, we assume that there exists some $0<\tau<\infty$ such that
\begin{align}
& \sup_{k \in \mathbb Z} \left|
\frac{\lambda_k}{w_k} \right| \leq \tau. \label{Progwohldef}
\end{align}
\end{Assumption}

The last assumptions ensures, that the linear prediction of $X$ with respect to $W$ is well defined.

In principle, if they were available, IV estimation would be based on the optimal instrument $\tilde W$ defined by
\begin{align*}
\tilde W=\Gamma_{WX}\Gamma_{W}^{-1} W  
= \sum_{k \in \mathbb Z} \frac{\overline{c_k}}{w_k} \langle W,\phi_k
\rangle \phi_k
\end{align*}
and the eigenvalues $(\lambda_k)_{k\in\mathbb N}$ of the corresponding cross covariance operator $\Gamma_{\tilde W X}$. However, this is usuall not the case and the optimal instrument respectively the corresponding eigenvalues of the cross covariance operator have to be estimated. Note, that $\tilde W$ could be exactly computed from $X$ and $W$, if the (cross) covariance operators were known and remark that $\lambda_k=\frac{|c_k|^2}{w_k}\leq x_k$ for all $k\in\mathbb Z$.

In the following, let $\{(X_i,W_i,Y_i)\}_{i=1,\ldots,n}$ be independent and identically distributed (i.i.d.)
copies of $(X,W,Y)$ and suppose \eqref{eq:exo} is valid. Then, we can consistently estimate the unknown slope function $\beta$ due to \cite{Joh13} and \cite{Joh16} in two different ways. For this purpose, let $(\alpha_n)_{n\in\mathbb N}$ be a sequence of regularization parameters such that $\alpha_n>0$ for all $n\in\mathbb N$ and $\lim_{n\to\infty}\alpha_n=0$. To simplify notation, we will write $\alpha$ for the regularization keeping in mind that it still depends on $n$. Since the covariance operators and therefore the corresponding eigenvalues are unknown, they have to be estimated in a first step.  Further, let $\Gamma_{WX,n},\Gamma_{X,n},\Gamma_{W,n}:\, L_2([0,1]) \to L_2([0,1])$ denote the empirical versions of $\Gamma_{WX},\Gamma_X$ and $\Gamma_W$, respectively, defined by
\begin{align*}
\Gamma_{WX,n} f := \frac1n \sum_{i=1}^n  \langle W_i ,f\rangle X_i,	\quad	\Gamma_{X,n }f &:= \frac1n \sum_{i=1}^n  \langle X_i,f\rangle X_i,	\quad	\text{and}	\quad	\Gamma_{W,n }f &:= \frac1n \sum_{i=1}^n  \langle W_i,f\rangle W_i
\end{align*}
for $f\in L_2([0,1])$. These estimators as well as the deduced estimators
\begin{align*}
&\hat w_k := \frac1n \sum_{i=1}^n |\langle W_i, \phi_k \rangle|^2, \,
&\hat x_k := \frac1n \sum_{i=1}^n |\langle X_i, \phi_k \rangle|^2,\\
&\hat c_k := \frac1n \sum_{i=1}^n \langle \phi_k, X_i
\rangle \langle W_i,\phi_k \rangle,\,
&\hat \lambda_k := \frac{|\hat c_k|^2}{\hat w_k}I\{\hat w_k \geq
\alpha\}
\end{align*}
for the eigenvalues $w_k$, $x_k$, $c_k$ and $\lambda_k$, respectively, are consistent for all $k\in\mathbb Z$. Hence, observations of the optimal linear instrument $\widetilde W$ can be estimated by
\begin{align*}
\widetilde W_{n,i} :=  \sum_{k \in \mathbb Z} \frac{\overline{\hat
c_k}}{\hat w_k}I\{ \hat w_k \geq \alpha \} \langle W_i,\phi_k \rangle
\phi_k, \quad i=1,\ldots,n,
\end{align*}
and the corresponding cross covariance operator by 
\begin{align}
\widetilde \Gamma_n := \frac1n \sum_{i=1}^n \langle\widetilde W_{n,i},\cdot\rangle X_i
    = \frac1n\sum_{k \in \mathbb Z} \frac{\overline{\hat c_k}}{\hat w_k} I\{
      \hat w_k \geq \alpha \}\sum_{i=1}^n \langle \cdot, X_i \rangle
      \langle W_i, \phi_k \rangle \phi_k. \label{tildegamman}
\end{align}
This allows to construct the IV-based estimator $\hat \beta_{IV} $ of the slope function $\beta$ defined by
\begin{equation}
\hat \beta_{IV} := \sum_{k \in \mathbb Z} \frac{\hat
g_k}{\hat \lambda_k} I \{\hat \lambda_k \geq \gamma_k^\nu \alpha
\}\phi_k=\sum_{k \in \mathbb Z} 
\frac{\frac1n \sum_{i=1}^n \langle W_i, \phi_k \rangle Y_i}{\hat c_k}
I\{\hat \lambda_k \geq \gamma_k^\nu \alpha \} I\{ \hat w_k
\geq \alpha \} \phi_k,	\label{eq:betaivdach}
\end{equation}
where
\begin{align*}
\hat g_k &= \frac{1}{n} \sum_{i=1}^n Y_i \langle \tilde
W_{n,i},\phi_k \rangle \nn .
\end{align*}
As shown in \cite{Joh16}, under Assumptions \ref{ass:secstat} and \ref{ass:eigenvalues} the estimator $\hat \beta_{IV}$ is consistent under the exogeneity assumption \eqref{eq:exo} as well as under endogeneity of \eqref{eq:endo}. 
In contrast, again under Assumptions \ref{ass:secstat} and \ref{ass:eigenvalues}, the estimator
\begin{equation}
\hat \beta = 
\sum_{k \in \mathbb Z} 
\frac{\frac1n \sum_{i=1}^n \langle X_i, \phi_k \rangle Y_i}{\hat x_k} I \{ \hat\lambda_k \geq \alpha \gamma_k^\nu \} 
\phi_k  \label{eq:betadach}
\end{equation}  
is only consistent under the exogeneity assumption (\ref{eq:exo}) (see \cite{Joh13}) and inconsistent under endogeneity of (\ref{eq:endo}). Note, that in comparison to the original definition in \cite{Joh13}, for $\hat \beta$, we
use the same indicator function $I \{ \hat\lambda_k \geq \alpha \gamma_k^\nu \}$ as in $\hat\beta_{IV}$. It turned out, that the tests perform better if the same regularization is used in both estimators although it might not be the best choice for estimating $\beta$ by $\hat\beta$ under assumption \eqref{eq:exo}.

Based on the two estimators (\ref{eq:betaivdach}) and (\ref{eq:betadach}), we construct the test statistic as
\begin{align}
T_n &= \frac1n \sum_{i=1}^n 
\left| \left\langle \hat\beta_{IV}-\hat \beta , X_i \right\rangle \right|^2=\left\langle  
\hat\beta_{IV}-\hat \beta,
\Gamma_{X,n}\left( \hat\beta_{IV}-\hat \beta \right)
\right\rangle.
\end{align}
The last  {representation above} corresponds to the idea used in \cite{MSM05} to construct
a goodness-of-fit test. The equivalence of both approaches can be seen by using the singular value decomposition for the estimators and for the covariance operator.

\begin{Assumption}
\label{ass:regulari}
For the sequence of regularization parameters, we assume
\begin{align*}
& \alpha_n=\alpha >0 \; \forall\, n \in \mathbb N,
\  \alpha = o(1) \text{ and }\  \frac{1}{n \alpha^2} = o(1).
\end{align*}
\end{Assumption}

For the next results, different moment conditions for $X$, $W$ and $U$ are required. To simplify the notation, we introduce the following sets. In doing so, we assume, that all conditions on $X$ and $W$ mentioned above are fulfilled and define
\begin{align}
\mathcal F_{\eta}^m 
&:= \Big\{ (X,W)\Big|\sup_{k \in \mathbb Z} \e \left| \frac{\langle
X, \phi_k \rangle}{\sqrt{x_k}} \right|^m \leq \eta\text{ and }
\sup_{k \in \mathbb Z}\e \left| \frac{\langle W, \phi_k \rangle}{\sqrt{w_k}} \right|^m
\leq \eta\Big\},   \label{MengeF} \\
%%%%%%%%%%%%%%%%%%%%%%%%%%%%%%%%%%%%%%%%%%%%%%%%%%%%%%%%%%%%%%%%%%%%%%%%
\mathcal G_{\eta}^m 
&:= 
\Big\{ X\Big|\Gamma_X>0
\text{ and } \sup_{k \in
\mathbb Z} \e \left|
\frac{\langle X, \phi_k \rangle}{\sqrt{x_k}} \right|^m \leq \eta \Big\}.   \label{MengeG}
\end{align}
In the following, for an operator $\Delta$, we denote by $\Delta^\dag$ the regularized inverse of the operator, that is
\[\Delta^\dag=\sum_{k\in\mathbb Z}\frac1{\delta_k}I\{|\delta_k|>\alpha\gamma_k^\nu\langle \cdot,\phi_k\rangle\phi_k\}\] 
and we define
\begin{equation}
t_{n}^2 
:= \| (\tilde\Gamma_{X,n}^\dag - \Gamma_X^\dag) \Gamma_X \|_{HS}^2
= \sum_{k \in \mathcal K_n} \left( \frac{x_k w_k}{|c_k|^2} -1  \right)^2,
\label{tkn}
\end{equation}
where $\|\cdot\|_{HS}$ denotes the Hilbert-Schmidt norm and we set
\begin{equation*}
\mathcal K_n := \{ k \in \mathbb Z \mid \lambda_k\geq \alpha \gamma_k^\nu \}.
\end{equation*}

%%%%%%%%%%%%%%%%%%%%%%%%%%%%%%%%%%%%%%%%%%%%%%%%%%%%%%%%%%%%%%%%%%%%%%%%%%%%%%%%%%%%%%%%%%%%%5
Now, we are in a position to state an asymptotic result for the test statistic.
\begin{theo}\label{CLTteststatistik}
In model (\ref{eq:model}), under Assumptions \ref{ass:secstat}-\ref{ass:regulari}, let $\{(X_i,W_i,Y_i)\}_{i=1,\ldots,n}$ be i.i.d.~copies of $(X,W,Y)$ with $\left(X,W\right)\in \mathcal F_{\eta}^{128}$ and $\e|U|^{128} \leq\eta<\infty$. Furthermore, let $t_n\to\infty$ as $n\to\infty$, and
\begin{align*}
\frac1{t_n^4}\sum_{k\in\mathcal K_n}\left(\frac{x_kw_k}{|c_k|^2}-1\right)^4 = o(1),	\quad 
\sum_{k \in \mathbb Z} | \langle \beta, \phi_k \rangle |
\frac{x_k^{3/2} w_k}{|c_k|^2} < \infty,	\quad \sum_{k \in \mathbb Z} \frac{x_k^2w_k}{|c_k|^2} < \infty.
\end{align*}
Then, under $H_0$, we have
\[ \frac{n}{t_{n}}\left( T_n - \mathfrak B_n -\mathfrak R_n \right)
\stackrel{\mathcal D}{\to} \mathcal N(0, \mathfrak V), \]
where
\begin{align*}
\mathfrak B_n&=\frac {n}{2t_n}\langle\beta,\mu_X\rangle^2\sum_{k\in\mathbb Z}\left(\frac{\langle\mu_W,\phi_k\rangle}{c_k}-\frac{\langle\mu_X,\phi_k\rangle}{x_k}\right)^2x_kI\{\lambda_k\geq \alpha\gamma_k^\nu\},\\
\mathfrak R_n&=\frac{1}{n} 
\left( \sigma^2 + \sum_{m \in \mathbb Z} |\langle \beta, \phi_m \rangle |^2 x_m \right)
\sum_{k \in \mathbb Z} 
\left(\frac{x_k w_k}{|c_k|^2} -1 \right) I\{ \lambda_k \geq \alpha \gamma_k^\nu \},\\
\mathfrak V&=\left(\sigma^2 + \sum_{m \in \mathbb Z} | \langle \beta, \phi_m \rangle|^2 x_m \right)^2.
\end{align*}
Additionally, if $X$ is centered, that is, $E[X(t)]=0$ for all $t\in[0,1]$, we have $\mu_X=0$ leading to $\mathfrak B_n=0$.
\end{theo}
%Note, that $\mathfrak B_n$ is 0 when $X$ is centered.

\bigskip

\begin{proof}
For the sake of simplicity, we assume, that $X$ is centered. If not, the additional bias term has to be taken into account as well as stated in the assertion of the theorem. We give a short overview of the proof. The used propositions and lemmas are stated and proven in the appendix. For the employed decomposition of the test statistic, we need several (modified) correlation operators of the instruments and $X$. We define $\mathcal U_n$, $\Delta_{W,n}:L_2([0,1]) \to \mathbb R$ by
\begin{align*}
\mathcal U_n f=\frac{1}{n} \sum_{i=1}^n (W_i \otimes U_i)f	\quad	\text{and}	\quad	\Delta_{W,n}f =\frac{1}{n} \sum_{i=1}^n (W_i \otimes Y_i)f,
\end{align*}
and set
\begin{align*}
\widetilde{\mathcal U}_n 
&:=  
\frac{1}{n} \sum_{i=1}^n \langle \cdot, \tilde W_i \rangle U_i
= 
\frac{1}{n} \sum_{k \in \mathbb Z} \frac{\overline{c_k}}{w_k}\sum_{i=1}^n \langle \phi_k, W_i \rangle \langle \cdot, \phi_k \rangle U_i, \\
\widehat{\widetilde{\mathcal U}}_n 
&:=  
\frac{1}{n} \sum_{i=1}^n \langle \cdot, \tilde W_{n,i} \rangle U_i
= 
\frac{1}{n} \sum_{k \in \mathbb Z} \frac{\overline{\hat c_k}}{\hat w_k}
\Iwkd \sum_{i=1}^n \langle \phi_k, W_i \rangle \langle \cdot, \phi_k
\rangle U_i.      
\end{align*} 
For the test statistic, the following decomposition holds
\begin{align}
\frac{n}{t_n} T_n 
&=
\frac{1}{t_n} \sum_{j=1}^n 
\left| \left\langle 
T_{n,1} + T_{n,2} + T_{n,3}, X_j
\right\rangle \right|^2
+
\frac{1}{t_n} \sum_{j=1}^n 
\langle 
T_{n,1} + T_{n,2}+ T_{n,3}, X_j \rangle
\langle X_j, R_n\rangle \nonumber\\
&\phantom{=}
+
\frac{1}{t_n} \sum_{j=1}^n 
\langle 
X_j, T_{n,1} + T_{n,2} + T_{n,3} \rangle
\langle R_n, X_j \rangle
+
\frac{1}{t_n} \sum_{j=1}^n 
\left| \left\langle 
R_n, X_j
\right\rangle \right|^2,\label{ZerlTeststat}
\end{align}
where
\begin{align*}
T_{n,1} 
&= 
\left(\tilde{\Gamma}_n^\dag \hat{\tilde{\mathcal U}}_n -
\Gamma_{X,n}^\dag\mathcal U_{X,n} \right)
-
\hat\Pi_{\mathcal K_n}\left(\tilde{\Gamma}^\dag \tilde{\mathcal U}_n - \Gamma_{X}^\dag
\mathcal U_{X,n}\right)\\
%%%%%%%%%%%%%%%%%%%%%%%%%%%%%%%%%%%%%%%%%%%%%%%%%%%%%%%%%%%%%%%%%%%%%%%%%%%%%%%%%%%%%%%
T_{n,2} 
&=
\left( \tilde{\Gamma}_n^\dag \tilde{\Gamma}_n - \Gamma_{X,n}^\dag \Gamma_{X,n}\right)
\beta  - \hat\Pi_{\mathcal K_n} A_n\\
%%%%%%%%%%%%%%%%%%%%%%%%%%%%%%%%%%%%%%%%%%%%%%%%%%%%%%%%%%%%%%%%%%%%%%%%%%%%%%%%%%%%%%%
T_{n,3} 
&=
\hat\Pi_{\mathcal K_n}
\left(\tilde{\Gamma}^\dag \tilde{\mathcal U}_n - \Gamma_{X}^\dag
\mathcal U_{X,n} + A_n\right) 
- 
\left(\tilde{\Gamma}^\dag \tilde{\mathcal U}_n - \Gamma_{X}^\dag
\mathcal U_{X,n} + A_n\right) \\
%%%%%%%%%%%%%%%%%%%%%%%%%%%%%%%%%%%%%%%%%%%%%%%%%%%%%%%%%%%%%%%%%%%%%%%%%%%%%%%%%%%%%%%
R_{n}
&=
\tilde{\Gamma}^\dag \tilde{\mathcal U}_n-\Gamma_{X}^\dag \mathcal U_{X,n} + A_n
\end{align*}
and
\begin{equation}
A_n = 
\frac{1}{n} \sum_{i=1}^n \sum_{k \in \mathbb Z} 
D_{i,k}
I\{ \lambda_k \geq \alpha \gamma_k^\nu \}
\sum_{\substack{m \in \mathbb Z, \\ |m| \neq |k|}} S_{i,m} \phi_k. \label{An}
\end{equation}
When subtracting $\mathfrak R_{n}$, the last term in \eqref{ZerlTeststat} can be further decomposed to get 
\begin{equation*}
\frac{1}{t_n} \sum_{j=1}^n 
\left| \left\langle 
R_n, X_j
\right\rangle \right|^2 -\frac{n}{t_n}\mathfrak R_{n}
=
\frac{n}{t_n} R_{n,3} + \frac{n}{t_n} \left( R_{n,2} - \mathfrak R_n \right)
+ \frac{n}{t_n} \left( R_{n,1} + R_{n,4} + R_{n,5} \right),
\end{equation*}
where $R_{n,i}$, $i=1,\ldots,5$ are defined in the appendix. There, we will also see that
\begin{align*}
%%%%%%%%%%%%%%%%%%%%%%%%%%%%%%%%%%%%%%%%%%%%%%%%%%%%%%%%%%%%%%%%%%%%%%%%%%%%%%%%%%%%%
\frac{n}{t_n} R_{n,3}&{=}
\frac{1}{nt_n} \sum_{k \in \mathbb Z} x_k I\{ \lambda_k \geq \alpha \gamma_k^\nu
\}
\sum_{\substack{i,j=1,\\ i \neq j}}^n 
D_{i,k} 
\Big( \sigma U_i + \sum_{\substack{ m \in \mathbb Z, \\ |m| \neq |k|}} S_{i,m}\Big)
\overline{D_{j,k}} 
\Big( \sigma U_j + \sum_{\substack{ m \in \mathbb Z, \\ |m| \neq |k|}} \overline{S_{j,m}}\Big)
 \nn %\\
%%%%%%%%%%%%%%%%%%%%%%%%%%%%%%%%%%%%%%%%%%%%%%%%%%%%%%%%%%%%%%%%%%%%%%%%%%%%%%%%%%%%%
\end{align*} 
converges weakly due to Theorem \ref{CLTRn3} to a normal distribution with mean 0 and variance $\mathfrak V$, while all remaining terms are discussed to be asymptotically negligible using Proposition
\ref{Rn1245}, \ref{Tn1}, \ref{Tn2} and \ref{Tn3} together with standard estimation techniques for the mixed terms. With the Lemma of Slutsky, the assertion follows.
\end{proof}

To apply the above result for testing,  {the bias and variance term have to be estimated}. To this end, note that $\sigma^2$ can be  {consistently} estimated by
\begin{equation}{\hat\sigma}_n^2=\frac{1}{n}\sum_{i=1}^n(Y_i-\langle\hat\beta_{IV},X_i\rangle)^2\label{eq:sigmahat}\end{equation}
 {due to} the law of large numbers and since $\frac1n\sum_{i=1}^n\langle\beta-\hat\beta_{IV},X_i\rangle^2=o_P(1)$ by similar calculations as in the derivation of the asymptotic distribution of $T_n$.

\begin{cor}
\label{CLTteststatistikemp}
Suppose all assumptions of Theorem \ref{CLTteststatistik} hold. Then, under $H_0$, we have
\[ \frac{n}{\hat t_{n}}\frac{T_n -\hat{\mathfrak B}_n-\hat{\mathfrak R}_n}{\sqrt{\hat{\mathfrak V}_n}}
\stackrel{\mathcal D}{\to} \mathcal N(0,1), \]
where ${\hat\sigma}_n^2$ is defined in \eqref{eq:sigmahat} and
\begin{align*}
\hat{t}_n^2 &= 
\sum_{k \in \mathbb Z}
\left( \frac{\hat x_k \hat w_k}{|\hat c_k|^2} - 1 \right)^2
I\{ \hat \lambda_k \geq \alpha \gamma_k^\nu \}, \\
\hat{\mathfrak B}_n&=\frac {n}{2\hat t_n}\langle\hat\beta_{IV},\hat\mu_X\rangle^2\sum_{k\in\mathbb Z}\left(\frac{\langle\hat\mu_W,\phi_k\rangle}{c_k}-\frac{\langle\hat\mu_X,\phi_k\rangle}{x_k}\right)^2\hat x_kI\{\hat\lambda_k\geq \alpha\gamma_k^\nu\},\\
\hat{\mathfrak R}_n&=\frac1n
\left( {\hat\sigma}_n^2 + \| \Gamma_{X,n}^{1/2} \hat\beta_{IV} \|^2 \right)
\sum_{k \in \mathbb Z}
\left( \frac{\hat x_k \hat w_k}{|\hat c_k|^2} - 1 \right)
I\{ \hat \lambda_k \geq \alpha \gamma_k^\nu \},\\
\hat{\mathfrak V}_n&=\left( {\hat\sigma}^2 + \| \Gamma_{X,n}^{1/2} \hat\beta_{IV} \|^2 \right)^2.	\\
\end{align*}
\end{cor}
Using Corollary \ref{CLTteststatistikemp}, it is possible to construct a test for the null hypothesis
\begin{equation}
H_0:\ \E[X(t)U]=0\mbox{ for all }t\in[0,1]
\end{equation}
against
\begin{equation}
H_1:\ \E[X(t)U]\neq0\mbox{ for at least one }t\in[0,1].\label{nullhyp}
\end{equation}
For given size $\gamma\in(0,1)$, we can reject $H_0$ if
\begin{equation}
\frac{n}{\hat t_{n}}\frac{T_n -\hat{\mathfrak B}_n-\hat{\mathfrak R}_n}{\sqrt{\hat{\mathfrak V}_n}}>u_{1-\gamma}\label{eq:stat_center}
\end{equation}
where $u_{1-\gamma}$ denotes the $(1-\gamma)$-quantile of the standard normal distribution. That is, we get a one-sided test for $H_0$ against $H_1$. In the special case $\mu_X=0$ we can neglect the additional bias term which avoids the use of its plug-in estimator such that the test has the simpler structure $I\{n(T_n -\hat{\mathfrak R}_n)/{{\hat t_{n}}\sqrt{\hat{\mathfrak V}_n}}>u_{1-\gamma}\}$.

\begin{cor}
Suppose all assumptions of Corollary \ref{CLTteststatistikemp} hold. Then, under the alternative $H_1$, the test constructed in \eqref{eq:stat_center} is consistent.
\end{cor}

\begin{proof}
We only consider the special case $\mu_X=0$ here. The general case can be proven by similar arguments. Under $H_1$, $\hat\beta$ is not consistently estimating $\beta$ such that it converges in probability to $\beta+b$ for some $b\in L^2[0,1]$ with
\begin{equation*}
b=\sigma\sum_{k\in\mathbb Z}\frac{E[U_1\langle X_1,\phi_k\rangle]}{x_k}I\{\lambda_k\geq \alpha\gamma_k^\nu\}\phi_k(t),
\end{equation*}
which is in general not equal to $0\in L^2[0,1]$ under endogeneity (by the continuity imposed in Assumption \ref{ass:secstat}). Hence, we have
\begin{align*}
T_n&=\frac1n\sum_{i=1}^n|\langle\hat\beta_{IV}-(\hat\beta-b),X_i\rangle|^2-\frac2n\sum_{i=1}^n\langle\hat\beta_{IV}-(\hat\beta-b),X_i\rangle\langle b,X_i\rangle+\frac1n\sum_{i=1}^n|\langle b,X_i\rangle|^2\\
&=\frac1n\sum_{i=1}^n|\langle\hat\beta_{IV}-(\hat\beta-b),X_i\rangle|^2-O_p\left(\sqrt{\frac {t_n}{n}}\right)+O_p\left(1\right).
\end{align*}
The standardized version of the first part converges in distribution to a standard normal distribution by similar arguments as in Theorem \ref{CLTteststatistik} and Corollary \ref{CLTteststatistikemp} while the sum of the remainder terms multiplied with $\frac n{t_n}$ goes to infinity for $n\to\infty$. Consequently, we have
\begin{align*}
P\left(\frac{n}{\hat t_{n}}\frac{T_n -\hat{\mathfrak R}_n}{\sqrt{\hat{\mathfrak V}_n}}>u_{1-\gamma}\right)\to1
\end{align*}
for $n\to\infty$.
\end{proof}

{In practice, we do not know, if $\mu_X=0$ such that a naive application of the asymptotic test without estimating ${\mathfrak B}_n$ could result in wrong decisions. In addition, asymptotic tests based on plug-in methods as above usually exhibit a smaller power compared to other methods. This is due to the additional estimation step. The bootstrap version of the test discussed in the next section is expected to have better finite sample behavior, since it is not {required} to estimate the unknown bias and variance. This has additionally the effect, that we need not distinguish between the cases $\mu_X=0$ and $\mu_X\neq 0$ which is a clear advantage of the bootstrap test.}

\section{Bootstrap Consistency}

%Consistency of a Bootstrap Version
\label{se:boot}
In this section, we use  {residual-based} bootstrap procedures to estimate the distribution of $$\frac{n}{t_{n}}\left( T_n - \mathfrak B_n -\mathfrak R_n \right)$$ under the null of exogeneity . To this end, we first estimate the residuals from the original data set and define
\begin{align*}
\hat U_i=Y_i-\langle\hat\beta_{IV},X_i\rangle,	\quad	i=1,\ldots,n,
\end{align*}
where we use the IV-based estimator, because it is consistent under the null hypothesis as well as under the alternative. However, using the classical estimator $\hat \beta$ would also result in a proper bootstrap scheme to approximate the distribution of the test statistics under the null of exogeneity, since the independence of error and regressor in the bootstrap sample is achieved by the  (fixed-design) bootstrap procedure itself. However, to get bootstrap data that mimics the true distribution {under the null hypothesis of exogeneity given the} original sample as close as possible, the IV-based estimator turns out to be more natural and performs better in simulations. In the sequel, different versions of residual-based bootstraps are considered. All bootstrap methods will follow these steps 

\begin{description}
\item[Step 1.)] Given i.i.d.~observations $(X_i,W_i,Y_i)$, $i=1,\ldots,n$, we generate a bootstrap sample $(X_i,W_i,Y_i^*)$, $i=1,\ldots,n$, by
\begin{align*}
Y_i^*=\langle\hat\beta_{IV},X_i\rangle+U_i^*,
\end{align*}
where  {the bootstrap errors} $U_i^*$ are generated from  {the residuals} $\hat U_1,\ldots,\hat U_n$ in such a way that the conditional independence of $U_i^*$ and $(X_i,W_i)$ is ensured. A thorough discussion, which types of bootstrap are appropriate  {in this sense} follows in the next subsection.
\item[Step 2.)] From $(X_i,W_i,Y_i^*)$, $i=1,\ldots,n$, a bootstrap test statistic $T_n^*$ is calculated.
\item[Step 3.)] Repeat Steps 1.) and 2.) $B$ times, where $B$ is large, to get bootstrap realizations $T_n^{*,1},\ldots,T_n^{*,B}$ of the test statistic and denote by $q_{1-\gamma}^*=T_n^{*,(\lfloor B(1-\gamma)\rfloor)}$ the corresponding empirical $(1-\gamma)$-quantile.
\end{description}

As the bootstrap errors are generated such that conditional independence of $U_i^*$ and $(X_i,W_i)$ is ensured, the bootstrap automatically adopts the exogeneity assumption. For the naive (Efron-type)residual bootstrap, this is trivially the case, because the bootstrap errors are drawn independently with replacement from the residuals, and for the wild bootstrap, since suitable bootstrap multiplier variables $V_i$ will also be drawn independently from $X_i$ and $W_i$.

\begin{theo}
\label{CLTboot}
Under the assumptions of Theorem \ref{CLTteststatistik} let $\mathcal S_n=\{(X_i,W_i,Y_i)\}_{i=1,\ldots,n}$ be a set of i.i.d.~copies of $(X,W,Y)$ with $\left( X,W \right) \in \mathcal F_{\eta}^{128}$ and $\e|U|^{128} \leq\eta<\infty$ and let $(t_n)_{n\in \mathbb N}$ from \eqref{tkn} fulfill $\lim_{n \to \infty} t_n = \infty$. Additionally, suppose that
\begin{align*}
& \frac1{t_n^4}\sum_{k\in\mathcal K_n}\left(\frac{x_kw_k}{|c_k|^2}-1\right)^4 =o(1),	\quad	\sum_{k \in \mathcal K_n} 
\left(x_k^2 \mathrm E |\langle \beta - \hat \beta_{IV}, \phi_k \rangle|^{4}\right)^{1/4}\frac{x_k^4 w_k^4}{|c_k|^8} = O(1),	\\
& \sum_{k \in \mathbb Z} 
\frac{x_k w_k^{1/2}}{|c_k|} < \infty,	\quad	\text{and}	\quad
\frac{1}{t_n} 
\sum_{k \in \mathcal K_n} 
\frac{x_k^{3/2} w_k}{|c_k|^2} = O(1) 
\end{align*} 
hold. Then, under both $H_0$ and $H_1$, we have
\begin{align*}
\sup_{t \in \mathbb R}
\left| P \left( \frac{n}{t_n} \left(T_n^* -\mathfrak B_n^* - \mathfrak R_n^* \right) \leq t \mid \mathcal S_n \right) 
- P_{H_0} \left( \frac{n}{t_n} \left(T_n -\mathfrak B_n - \mathfrak R_n \right) \leq t 
\right) \right|
\stackrel{\mathbb P}{\longrightarrow} 0,
\end{align*}
where $\mathfrak B_n^*$ and $\mathfrak R_n^*$ denote the bootstrap versions of $\mathfrak B_n$ and $\mathfrak R_n$ defined in Theorem \ref{CLTRn3} and $P_{H_0}$ is the distribution of $\frac{n}{t_n} \left(T_n -\mathfrak B_n - \mathfrak R_n \right) $ under $H_0$.
\end{theo}

Based on this result, we can again construct a one-sided test for the hypotheses \eqref{nullhyp} which rejects the null hypothesis if $T_n>q_{1-\gamma}^*$ from Step 3 since $T_n,T_n^\ast\geq 0$ and both asymptotically have the same bias and variance.

\section{Generalization to other estimators and measuring goodness-of-fit}
While the above results are stated for the spectral-cut-off estimators as proposed in \cite{Joh13} and \cite{Joh16}, it is also possible to derive analogue results for other types of estimators like cut-off as in \cite{MSM05} or ones based on Tikhonov or ridge-type regularization. A quite general approach is given in \cite{Car06} with a sequence of regularization functions $f_n:[c_n,\infty)\to\mathbb R_0^+$ such that $f_n$ is decreasing on $[c_n,2z_1-z_2]$ where $(z_j)_{j\in\mathbb Z}$ are the eigenvalues of the relevant covariance operator and $(c_n)_{n\in\mathbb N}$ is a decreasing sequence of positive values with $c_n<z_1$. Furthermore $\lim_{n\to\infty}\sup_{z\geq c_n}|zf_n(z)-1|=o(1/\sqrt n)$ and $f_n$ is differentiable on $[c_n,\infty)$ which replaces Assumption \ref{ass:regulari}. While the estimator $\hat\beta$ from (\ref{eq:betadach}) above does not completely fit this situation it is not neccessary to consider this modification if one is only interested in testing goodness-of-fit in an exogeneous model (\ref{eq:model}). For the sake of shorter notation we assume here $\mu_X\equiv0$. If $\tilde\beta$ denotes the estimator proposed in \cite{Car06} we obtain under Assumption \ref{ass:secstat} and the moment conditions in Theorem \ref{CLTteststatistik} the following result
\[\frac{n}{s_{n}}\left(\frac{1}{n}\sum_{i=1}^n\left|\left\langle\tilde\beta-\beta,X_i\right\rangle\right|^2 -\mathfrak R_n \right)
\stackrel{\mathcal D}{\to} \mathcal N(0, \mathfrak V)\]
with $s_n=\sum_{k\in\mathbb Z}x_k^4f_n^4(x_k)$, $\mathfrak V$ as in Theorem \ref{CLTteststatistik} and
\[\mathfrak R_n=\frac{1}{n} 
\left( \sigma^2 + \sum_{m \in \mathbb Z} |\langle \beta, \phi_m \rangle |^2 x_m \right)
\sum_{k \in \mathbb Z} 
f_n(x_k)=O\left(\frac1{\sqrt n}\right).\]
If is straight forward to also generalize the instrumental variable estimator to other regularization schemes. We get an estimator
\[\tilde\beta_{IV}=\sum_{k\in\mathbb Z}\hat g_kf_n(\hat\lambda_k)\]
and, if we are willing to assume exogeneity here, derive by the same arguments as above
\[\frac{n}{s_{n,IV}}\left(\frac{1}{n}\sum_{i=1}^n\left|\left\langle\tilde\beta_{IV}-\beta,X_i\right\rangle\right|^2 -\mathfrak R_n \right)
\stackrel{\mathcal D}{\to} \mathcal N(0, \mathfrak V)\]
with $s_{n,IV}=\sum_{k\in\mathbb Z}x_k^2\lambda_k^2f_n^4(\lambda_k)$, $\mathfrak V$ as in Theorem \ref{CLTteststatistik} and
\[\mathfrak R_n=\frac{1}{n} 
\mathfrak V^{1/2}
\sum_{k \in \mathbb Z} 
x_k\lambda_kf_n(\lambda_k)=O\left(\frac1{\sqrt n}\right).\]
The assumption of exegeneity is in this case not realistic because one would only use the instrumental variable estimator under endogeneity. Proving an analogue result under endogeneity is in principle possible but the proof differs in several points from the one presented here.\\

Using the estimators $\tilde\beta$ and $\tilde\beta_{IV}$ we can construct a test statistic similar to the one above. To this end we need a similar regularization scheme for both estimators. If we allow for a second argument in $f_n$ the estimators involved in the test above can also be written with $f_{n,1}(x_k,\lambda_k)=\frac1{x_k}I\{\lambda_k\geq\alpha\gamma_k^\nu\}$ for $\hat\beta$ and $f_{n,2}(x_k,\lambda_k)=\frac1{\lambda_k}I\{\lambda_k\geq\alpha\gamma_k^\nu\}$ for $\hat\beta_{IV}$ and it is straight forward to generalize them at least to regularisation functions of type $f_{n,1}(x_k,\lambda_k)=g_1(x_k,\lambda_k)\tilde f_n(\lambda_k)$ respectively $f_{n,2}(x_k,\lambda_k)=g_2(x_k,\lambda_k)\tilde f_n(\lambda_k)$. Under Assumption \ref{ass:secstat}, the moment assumptions of Theorem \ref{CLTteststatistik} and certain regularity conditions we derive under the null hypthesis

\[ \frac{n}{\tilde t_{n}}\left( T_n - \mathfrak B_n -\mathfrak R_n \right)
\stackrel{\mathcal D}{\to} \mathcal N(0, \mathfrak V), \]
where 
$$\tilde t_n=\sum_{k\in\mathbb Z}(\lambda_kg_2^2(x_k,\lambda_k)-2\lambda_kg_1(x_k,\lambda_k)g_2(x_k,\lambda_k)+x_kg_1(x_k,\lambda_k))^2{\tilde f_n}^2(\lambda_k) ,$$
$\mathfrak V$ as in Theorem \ref{CLTteststatistik} and
\begin{align*}
\mathfrak R_n&=\frac{1}{n} 
\mathfrak V^{1/2}
\sum_{k \in \mathbb Z} 
\left(\lambda_kg_2^2(x_k,\lambda_k)-2\lambda_kg_1(x_k,\lambda_k)g_2(x_k,\lambda_k)+x_kg_1(x_k,\lambda_k)\right) \tilde f_n(\lambda_k).
\end{align*}

For all results presented in this section it is again straight forward to derive empirical versions and bootstrap results.

\section{Finite sample properties}
\label{se:finite}
In this section, we investigate the finite sample behavior of the tests proposed above under several degrees of endogeneity and for different slope functions. We generate our data from the model

\begin{align*}
X(t)= \left(t + \frac{1}{2}\right) Z_{1},	\quad	W(t) = \left(t + \frac{1}{2}\right) Z_{2}+ H   
\end{align*} 
and
\[ Y =
\frac{1}{p+1} \sum_{l=0}^p X(l/p+1)\cdot \beta(l/p) + \sigma\cdot \varepsilon, \]
for some $p$ large enough to approximate the integral sufficiently well. To control all correlations in the model, we generate i.i.d.~copies of
\begin{equation*}
\begin{pmatrix} Z_1 \\ Z_2 \\ U \end{pmatrix}
\sim \mathcal N_3 \left(
\begin{pmatrix} 0 \\ 0 \\ 0 \end{pmatrix}
,
\begin{pmatrix}
3 & \nu\sqrt 6 & \rho \sqrt 3 \\
 \nu\sqrt 6 & 2 & 0  \\
 \rho \sqrt 3 & 0 & 1  \\
\end{pmatrix} \right) 
\end{equation*}
with $corr(Z_{1},Z_{2}) = \nu, \; corr(Z_{1},U) = \rho $, see \cite{Won96}. The
random variable $H$ is uniformly distributed on $(-1/2,1/2)$ and independent of
$(Z_1,Z_2,\varepsilon)^\prime$. The parameter $\rho$ controls the severity of endogeneity
(if $\rho=0$ we are in the exogenous case, i.e.~under the null $H_0$) and $\nu$ the
strength of the instrument $W$. The standard deviation is assumed to be
$\sigma=7/5$.
In the following, as illustrated in Figure \ref{graphik_betas}, we will use three different slope functions $\beta_1,\beta_2$ and $\beta_3$ defined by
\begin{align}
\beta_1(t) &= \sin(4 \pi t) + \frac12 \sin\left(8 \pi t\right) + \frac{1}{7} \sin\left(20
\pi t\right) \nn, \\
\beta_2(t) &= \frac{2}{\pi} \arcsin(\cos(2 \pi t)) \nn, \\
\beta_3(t) &= \sum_{n \in \mathbb Z}
              \int_{\mathbb R}
                 r_n(s) k_{n,h}(t-s) \, ds,  
\label{betas}
\end{align}
\begin{figure}
\includegraphics[width=\linewidth]{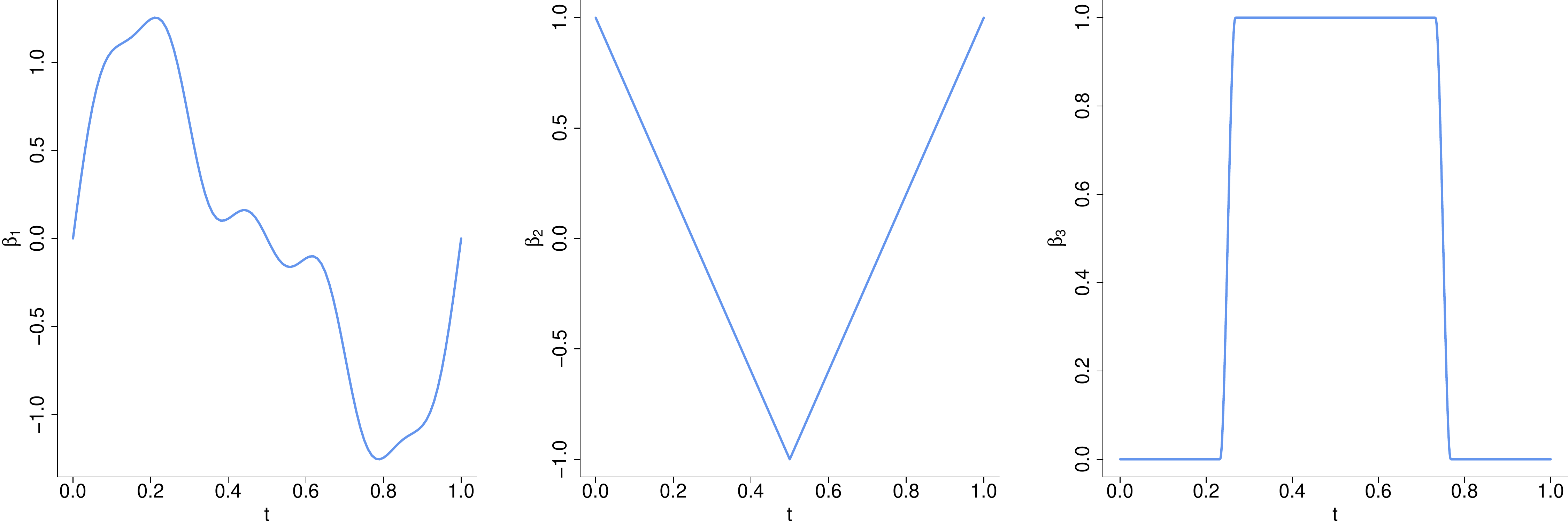}
\caption{\label{graphik_betas} Slope functions $\beta_1,\beta_2$ and $\beta_3$ as defined in \eqref{betas}.}
\end{figure}
where in $\beta_3$, $r_n(t) = I_{\big\{ n+\frac14, n + \frac34 \big\}} (t)$ and $k_{n,h} (t) = \frac{1}{h} k_n
(\frac{t}{h})$ with $$k_n (t) = \frac{1}{C} \exp \left( -\frac{1}{1- (t-2n)^2} \right)
                           I_{(-1+2n, 2n+1)} (t)$$ 
and $C = \int_{\mathbb R} k_0(s) \, ds$. For all simulations, we generate 1000  {Monte Carlo} realizations and use $B=500$ bootstrap replications.

Besides an Efron-type residual-based bootstrap, which draws the bootstrap errors $U_i^*$, $i=1,\ldots,n$ independently with replacement from the residuals $\hat U_1,\ldots,\hat U_n$, we consider also several versions of a residual-based wild bootstrap, where
\begin{align*}
U_i^*=V_i\hat U_i,	\quad i=1,\ldots,n
\end{align*}
and the $V_i$'s are i.i.d.~with $\E[V_1]=0$ and
$\E[V_1^2]=1$ and independent of $(X_i,W_i,Y_i)_{i=1,\ldots,n}$. We consider different choices  {for the distribution} of the $V_i$'s as commonly used in the literature, see e.g.~\cite{Mam93},
\begin{align}
a)&\quad P\left(V_1=\frac{-
(\sqrt{5}-1)}{2}\right)=\frac{\sqrt{5}+1}{2 \sqrt{5}},\quad P\left(V_1=\frac{\sqrt{5}+1}{2}\right)=\frac{\sqrt{5}-1}{2 \sqrt{5}},\label{goldschn}\\
b)&\quad P(V_1=1)=0.5=P(V_1=-1),\label{radem}\\
c)&\quad V_1\sim \mathcal N(0,1).\label{norm}
\end{align} 

%Additionally, we simulate using a residual bootstrap, see e.g. \cite{GM10} where $U_i^\ast$, $i=1,\ldots,n$ are drawn with replacement from $\hat U_1,\ldots,\hat U_n$.

\begin{figure}
\includegraphics[width=\textwidth]{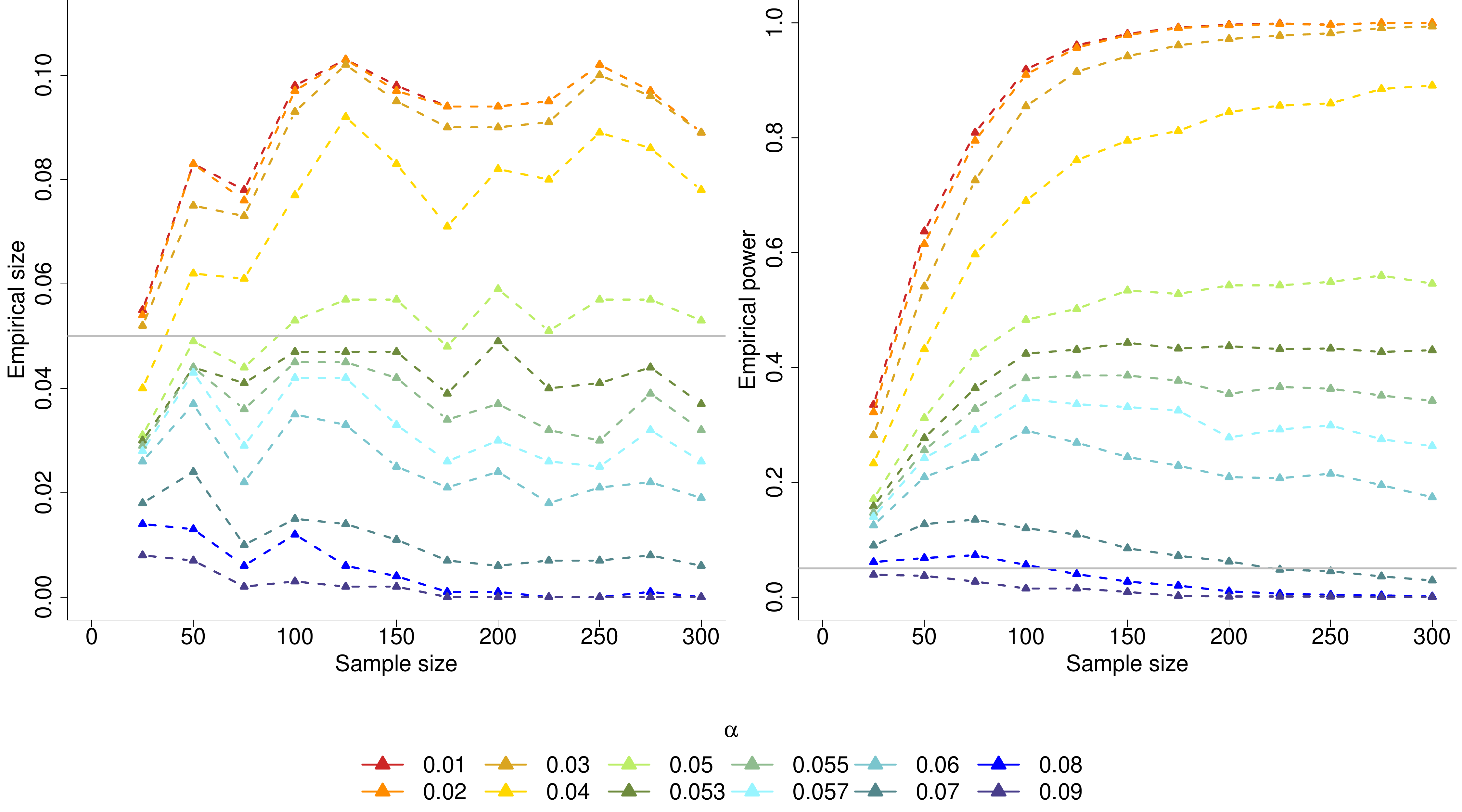}
\caption{\label{asympt_div_alpha}Empirical size and power of the asymptotic test for several choices of $\alpha$. The gray solid line shows the target level $\gamma=0.05$. The true slope parameter function is $\beta_1$.}
\end{figure}

In a first step we try to get an idea how to choose $\alpha$ and, in a next step how to choose $\mathcal
K_n$. To this end, we fix the degree of endogeneity with $\rho=0.4$ and the strength of
the instrument with $\nu=0.6$. In Figure \ref{asympt_div_alpha},
the results for the asymptotic test using $\beta_1$ as slope
parameter and different choices of $\alpha$ are shown. We see that the best results are obtained for $\alpha$ between $0.05$ and $0.055$. For smaller $\alpha$, the test does not hold the {prescribed} level, while for larger $\alpha$ the power is comparably small up to biased tests for $\alpha$ larger than $0.07$. Based on Figure \ref{asympt_div_alpha}, we can find a sequence of good choices for $\alpha$ depending on the sample size varying from $\alpha=0.04$ for $n=25$ to $\alpha=0.053$ for larger sample sizes up to $300$.
We see that the asymptotic test has only moderate power even for larger sample sizes. This is a well known effect with asymptotic tests using plug-in estimators. 

The way out is typically a bootstrap-based test. The results for the residual-based bootstraps proposed in Section \ref{se:boot} and again $\beta_1$ are shown in Figure \ref{bootstrap_div_alpha}.
\begin{figure}
\includegraphics[width=\textwidth]{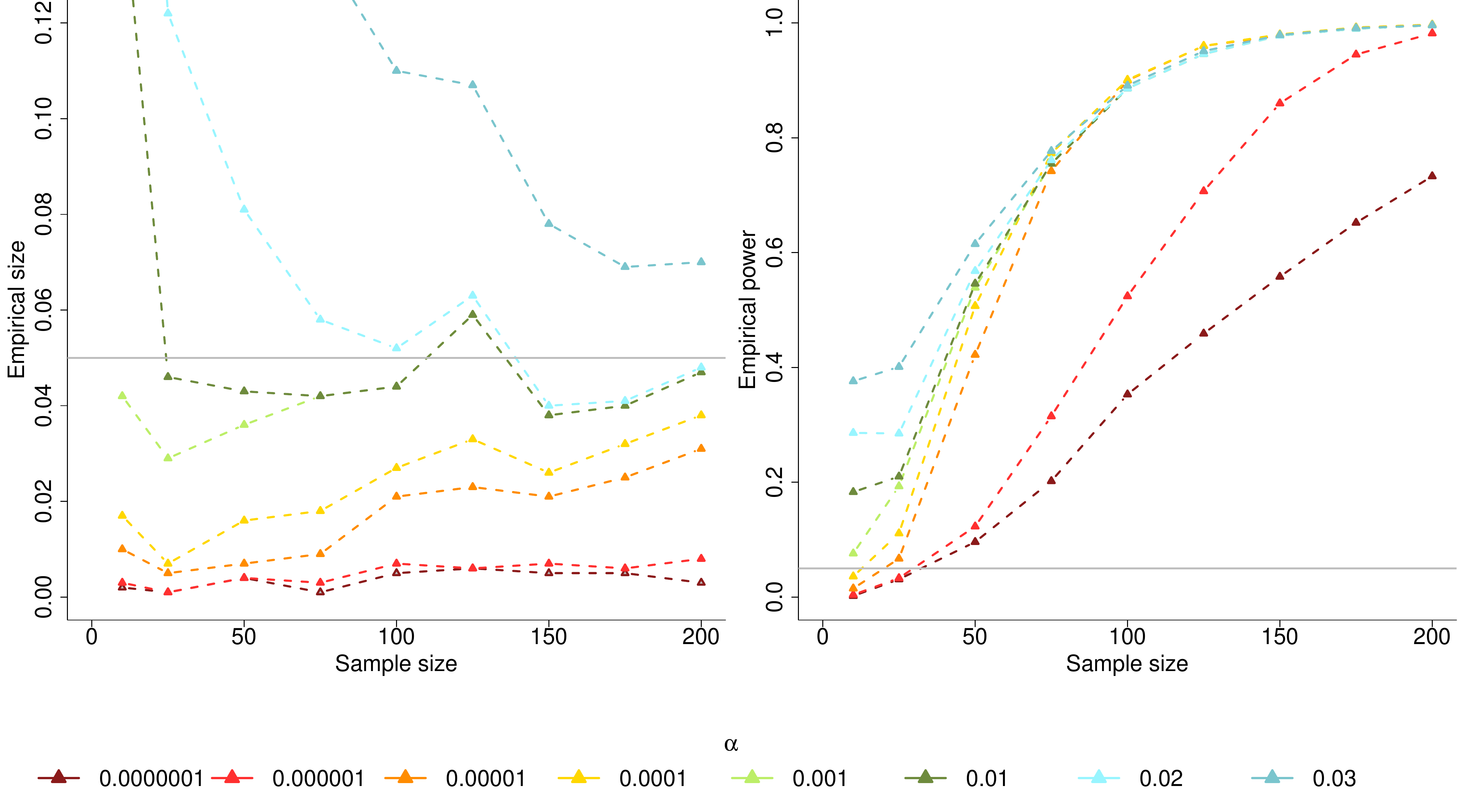}
\caption{\label{bootstrap_div_alpha}Empirical size and power of the bootstrap tests for several choices of $\alpha$. The gray solid line shows the target level $\gamma=0.05$. The true slope function is $\beta_1$.}
\end{figure}
It turns out, that the regularization parameter can be chosen  {considerably} smaller than
for the asymptotic test and the procedure is much more robust in choosing
$\alpha$.
Nearly all tests hold the size of $\gamma=0.05$ for larger sample sizes and the power increases with sample size for most choices of $\alpha$ up to a value close to 1 already for $n=300$.  Again we can get an idea of choosing a good $\alpha$ depending on the sample size which varies from $\alpha=0.01$ for $n=25,50$ to $\alpha=0.0001$ for $n=75,100,200$ and $300$.

Apparently all bootstrap procedures discussed in Section \ref{se:boot} perform comparably good which can be seen in Figure \ref{comp_boots} for a choice of $\alpha=0.0001$.

\begin{figure}
\includegraphics[width=\textwidth]{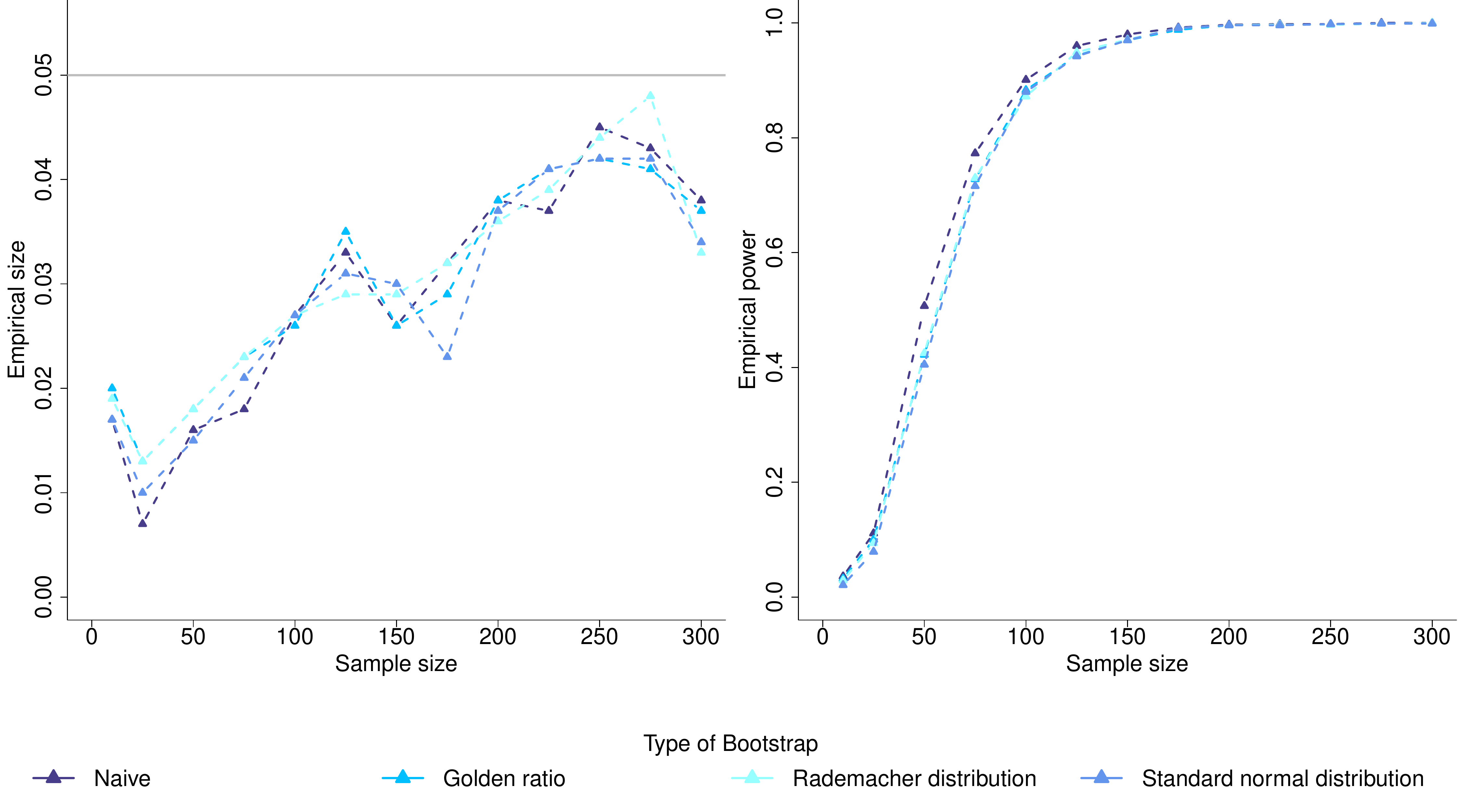}
\caption{\label{comp_boots}Empirical power and size of the bootstrap tests for several choices of $\alpha$. The gray solid line shows the predetermined level $\gamma=0.05$ of the test}
\end{figure}

Comparing the performance of the bootstrap test for different slope functions, we discover that in all models the bootstrap test holds the size $\gamma=0.05$ while we see in Table \ref{bootdivbeta} that the power is similarly good for all settings with only sligh disadvantages for the smoothed indicator function $\beta_3$.
\begin{table}[t]
 %\vspace{-5 mm}
\resizebox{\columnwidth}{!}{%
 \begin{tabular}{c|c|c|c|c|c|c|c|c|c|c|c|c} 
\multicolumn{13}{c}{$n$ }  \\
 & 25    & 50   & 75   & 100   & 125   & 150   & 175    & 200   & 225 & 250 &275 &300 \\
\hline
%\multirow{3}{*}{$\beta$} &
   $\beta_1$ & 0.111  & 0.507 & 0.773 & 0.901 & 0.960 & 0.980 & 0.992  & 0.997  &
    0.998 & 0.998 & 1 & 1 \\
   $\beta_2$ & 0.164  & 0.568 & 0.798 & 0.912 & 0.958 & 0.979 & 0.992  & 0.997   &
    0.999 & 0.998 & 1 & 1 \\
  $\beta_3$ & 0.255  & 0.560 & 0.733 & 0.853 & 0.904 & 0.961 & 0.978  & 0.990 & 0.993 &
    0.994 & 0.997 & 0.998\\
 \end{tabular}}
  \caption{Empirical power of the bootstrap tests for slope functions defined in \eqref{betas} using
  $\rho=0.4$, $\nu = 0.6$ and $\alpha = 0.0001$.}
  \label{bootdivbeta} 
 \end{table}
 Finally, we inspect the influence of the degree of endogeneity and the strength of the instrument on the performance of the test. In Figure \ref{bootdivrho}, we see that the power of the bootstrap test increases with increasing degree $\rho$ of endogeneity being already acceptable for $\rho=0.3$.
 \begin{figure}[h] 
\includegraphics[width=\linewidth]{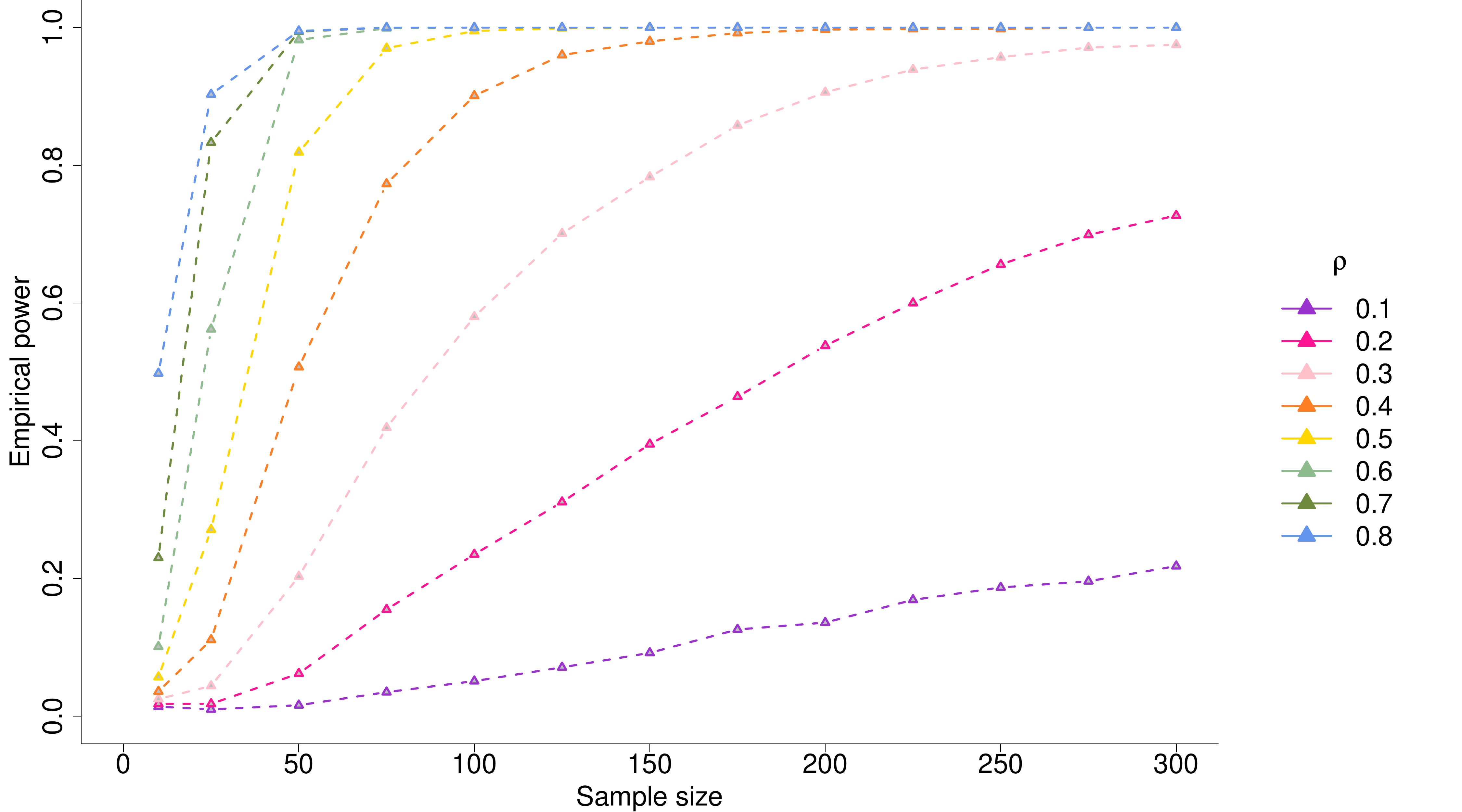}
\caption{Power of the bootstrap test for different degrees $\rho$ of endogeneity}
\label{bootdivrho}
\end{figure} 
Figure \ref{bootdivnu} shows, that the performance of the test is highly dependent on the strength of the instrument.
%Figure \ref{bootdivnu} shows, that it is not that easy to find a proper instrument for the instrumental variable estimator.
If the instrument is too weak, the power is
too low and the test does not hold the size. It turns out, that for the setting with slope function $\beta_1$, $\rho=0.4$ and $\alpha=0.0001$, the bootstrap test performs best for a strength of the instrument around $\nu=0.7$.
\begin{figure}[h]
\includegraphics[width=\linewidth]{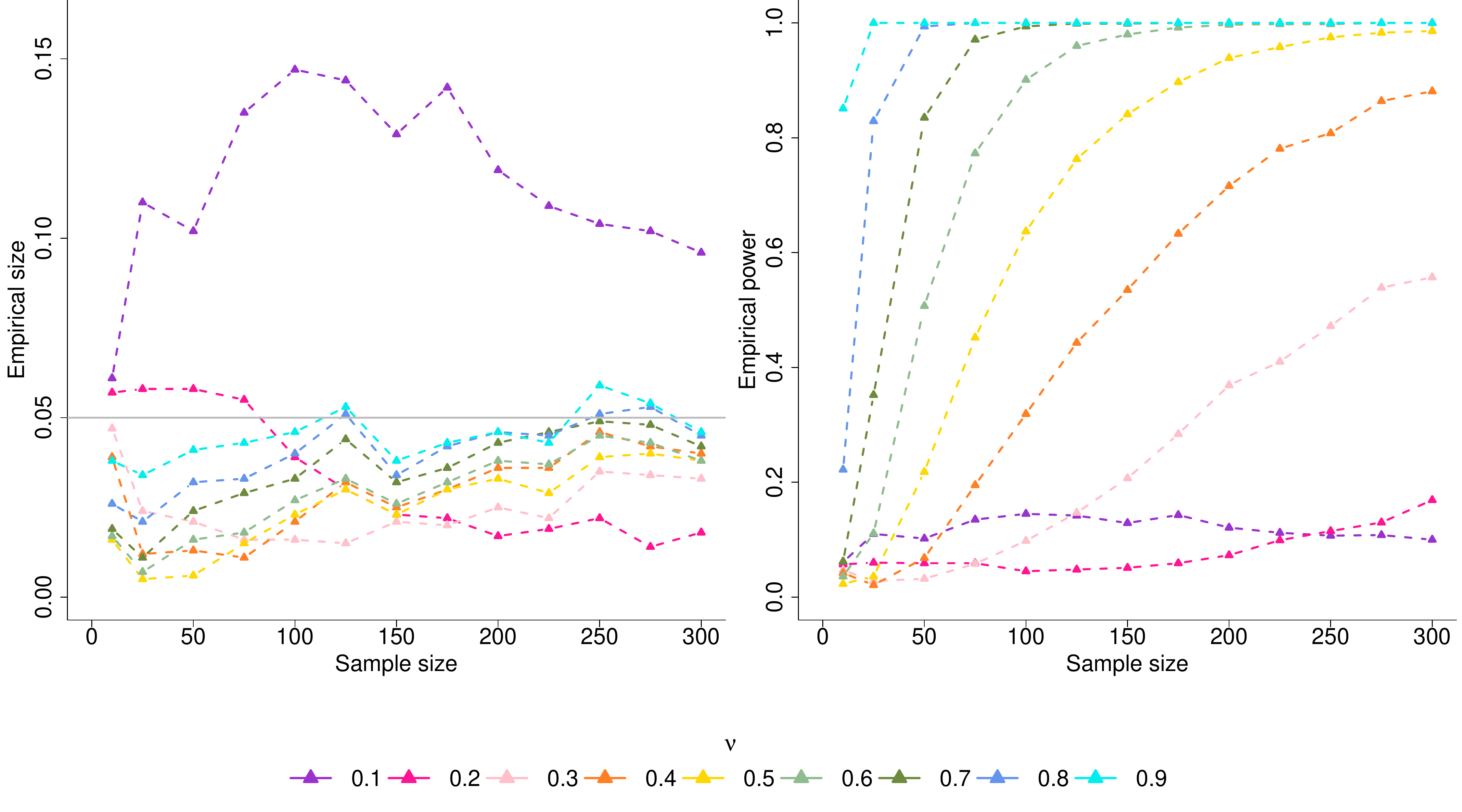}
\caption{Power and size of the bootstrap test for different strengths $\nu$ of the instrument}
\label{bootdivnu}
\end{figure}

\section{Concluding remarks}\label{se:out}
The underlying work is the first approach of testing for endogeneity in a functional regression setup by introducing a modified approach of the classical Hausman test in a multiple linear regression model. {This modification is required, because the $L_2$-distance of two slope function estimators in functional linear regression models are shown to have no proper limiting distribution. We prove asymptotic normality for the proposed modified Hausman-type test statistic, which allows for the construction of asymptotic tests for exogeneity. As the asymptotic test has several drawbacks such as many nuisance parameters, which are cumbersome to estimate, an additional bias term, which diverges when multiplied with the rate of convergence, and a high sensitivity to the choice of the regularization parameter, we propose suitable bootstrap versions of the test to approximate the null distribution. This avoids the additional estimation of nuisance parameters and turns out to be much more robust to the choice of the regularization parameter.} This behavior is demonstrated in a detailed simulation study. Topics of ongoing work are the choice of the instrument, a data driven choice of the regularization parameter and the transfer to other regression models.

%%%%%%%%%%%%%%%%%%%%%%%%%%%%%%%%%%%%%%%%%%%%%%
%% Single Appendix:                         %%
%%%%%%%%%%%%%%%%%%%%%%%%%%%%%%%%%%%%%%%%%%%%%%
\begin{appendix}\label{appA}

\section{Auxiliary Results for the Proof of Theorem \ref{CLTteststatistik}}
We assume for the sake of simplicity $E[X(t)]=E[W(t)]=0$ for all $t\in[0,1]$ and remember from Section \ref{se:modtest} the decomposition of the test statistic with
\begin{align}
%%%%%%%%%%%%%%%%%%%%%%%%%%%%%%%%%%%%%%%%%%%%%%%%%%%%%%%%%%%%%%%%%%%%%%%%%%%%%%%%%%%%%%
R_{n,1} &=
\frac{1}{n^2} \sum_{k \in \mathbb Z} 
\left( \hat x_k - x_k \right)
 \Big| \sum_{i=1}^n D_{i,k} I\{ \lambda_k \geq \alpha \gamma_k^\nu \} 
 \Big( \sigma U_i + \sum_{\substack{ m \in \mathbb Z, \\ |m| \neq |k|}}
 S_{i,m}\Big)\Big|^2 \nn \\
%%%%%%%%%%%%%%%%%%%%%%%%%%%%%%%%%%%%%%%%%%%%%%%%%%%%%%%%%%%%%%%%%%%%%%%%%%%%%%%%%%%%%%
R_{n,2} &=
%+
\frac{1}{n^2} \sum_{k \in \mathbb Z} 
x_k I\{ \lambda_k \geq \alpha \gamma_k^\nu \}
 \sum_{i=1}^n |D_{i,k}|^2 
 \Big| \sigma U_i + \sum_{\substack{ m \in \mathbb Z, \\ |m| \neq |k|}} S_{i,m}\Big|^2 \nn,
 \\
%%%%%%%%%%%%%%%%%%%%%%%%%%%%%%%%%%%%%%%%%%%%%%%%%%%%%%%%%%%%%%%%%%%%%%%%%%%%%%%%%%%%%
R_{n,3} &=
\frac{1}{n^2} \sum_{k \in \mathbb Z} x_k I\{ \lambda_k \geq \alpha \gamma_k^\nu
\}
\sum_{\substack{i,j=1,\\ i \neq j}}^n 
D_{i,k} 
\Big( \sigma U_i + \sum_{\substack{ m \in \mathbb Z, \\ |m| \neq |k|}} S_{i,m}\Big)
\overline{D_{j,k}} 
\Big( \sigma U_j + \sum_{\substack{ m \in \mathbb Z, \\ |m| \neq |k|}} \overline{S_{j,m}}\Big)
 \nn, \\
%%%%%%%%%%%%%%%%%%%%%%%%%%%%%%%%%%%%%%%%%%%%%%%%%%%%%%%%%%%%%%%%%%%%%%%%%%%%%%%%%%%%%
R_{n,4} &= \frac{1}{n^3} \sum_{\substack{k,l \in \mathbb Z,\\ |k| \neq |l|}}\sum_{j=1}^n 
\langle \phi_k,X_j \rangle
\langle X_j, \phi_l \rangle 
I\{ \lambda_k \geq \alpha \gamma_k^\nu \}
I\{ \lambda_l \geq \alpha \gamma_l^\nu \}\\&\phantom{\frac{1}{n^3} \sum_{\substack{k,l \in \mathbb Z,\\ |k| \neq |l|}}\sum_{j=1}^n 
\langle \phi_k,X_j \rangle}\times
%%%%%%%%%%%%%%%%%%%%%%%%%%%%%%%%%%%%%%%%%%%%%%%%%%%%%%%%%%%%%%%%%%%%%%%%%%%%%%%%%%%%% 
\sum_{i=1}^n 
D_{i,k} 
\Big( \sigma U_i + \sum_{\substack{ m \in \mathbb Z, \\ |m| \neq |k|}} S_{i,m}\Big)
\overline{D_{i,l}} 
\Big( \sigma U_i + \sum_{\substack{ m \in \mathbb Z, \\ |m| \neq |l|}}
\overline{S_{i,m}}\Big)\nn,
\\
%%%%%%%%%%%%%%%%%%%%%%%%%%%%%%%%%%%%%%%%%%%%%%%%%%%%%%%%%%%%%%%%%%%%%%%%%%%%%%%%%%%%%
R_{n,5}&={\frac{1}{n^3} \sum_{\substack{k,l \in \mathbb Z,\\ k \neq l}}\sum_{j=1}^n
\langle \phi_k,X_j \rangle } 
\sum_{\substack{i_1,i_2=1,\\ i_1 \neq i_2}}^n 
D_{i_1,k}  
\Big( \sigma U_{i_1} + \sum_{\substack{ m \in \mathbb Z, \\ |m| \neq |k|}} S_{i_1,m}\Big)
\ov{D_{i_2,l}} 
\Big( \sigma U_{i_2} + \sum_{\substack{ m \in \mathbb Z, \\ |m| \neq |l|}}
\ov{S_{i_2,m}}\Big)  \label{ZerlRn}
\end{align} 
and define
\begin{align}
D_{i,k,n} 
&= 
\frac{\langle W_i, \phi_k \rangle}{\hat c_k}
I\{ \hat w_k \geq \alpha \}
-
\frac{1}{\hat x_k} \langle X_i, \phi_k \rangle, \label{Dn}\\
D_{i,k} 
&= 
\left(\frac{\langle W_i, \phi_k \rangle}{c_k}
-
\frac{1}{x_k} \langle X_i, \phi_k \rangle \right)  \label{D},\\
S_{i,m} 
&=  
 \langle \beta, \phi_m \rangle \langle \phi_m, X_i \rangle. \label{S}
\end{align} 
The first result establishes the asymptotic distribution of the test statistic.

\begin{theo}
Under the assumptions of Theorem \ref{CLTteststatistik}, under the null hypothesis, and for $(X,W) \in \mathcal F_{\eta}^4$ and $E[X(t)]=E[W(t)]=0$ for all $t\in[0,1]$, we have
\[ \frac{n}{t_{n}}R_{n,3}
\Dkonv \mathcal N(0, \mathfrak V).\]
 
\label{CLTRn3}
\end{theo}

The remaining results are to show, that the remainder terms are negligible.

\begin{prop}\label{Tn1}
Let $\left( X,W \right)\in \mathcal F_{\eta}^{128}$ and $\mathrm E|U|^{128} \leq\eta<\infty$. Under the assumptions of Theorem \ref{CLTteststatistik}, we have
\begin{equation*}
\frac{1}{n} \sum_{j=1}^n 
\left| \left\langle
T_{n,1}
,X_j\right\rangle
\right|^2 = o_P \left( \frac{1}{n}\right).
\end{equation*}
\end{prop}

\begin{prop}\label{Tn2} 
Under the assumptions of Theorem \ref{CLTteststatistik} and if $\left( X,W \right) \in \mathcal F_{\eta}^{64}$ and $\mathrm E|U|^{64} \leq\eta< \infty$, we have
\begin{equation*}
\frac{1}{n} \sum_{j=1}^n 
\left| \left\langle T_{n,2}
, X_j
\right\rangle \right|^2 
= o_P \left( \frac{t_n}{n} \right).
\end{equation*}
\end{prop}

\begin{prop}\label{Tn3}
Under the assumptions of Theorem \ref{CLTteststatistik}, and if $\left( X,W \right) \in \mathcal F_{\eta}^{8}$ and $\mathrm E|U|^{8}\leq\eta< \infty$, we have
\begin{equation*}
\frac{1}{n} \sum_{j=1}^n 
\left| \left\langle T_{n,3}
, X_j
\right\rangle \right|^2 
= o_P \left( \frac{t_n}{n} \right).
\end{equation*}
\end{prop}

\begin{prop}\label{Rn1245}
Under the assumptions of Theorem \ref{CLTteststatistik}, and if
$\E|U|^{4} \leq\eta < \infty$ and $(X,W) \in \mathcal F_{\eta}^4$, we have
\begin{align*}
R_{n,1} &= o_P \left( \frac{1}{n} \right),	\quad	R_{n,2} = \mathfrak R_n +  o_P \left( \frac{t_n}{n} \right),	\quad \mathfrak R_n = o \left( \frac{1}{\sqrt{n}} \right),\\
R_{n,4} &= o_P \left( \frac{1}{n^{3/2}} \right),	\quad	R_{n,5} = o_P \left( \frac{1}{n} \right).
\end{align*}
\end{prop}
%%%%%%%%%%%%%%%%%%%%%%%%%%%%
\section{Auxiliary results}
The results in this section are used at several places in the proofs. They follow from Lemma A.1 in \cite{Joh16}.
\begin{lemma}
\label{Reihenkonvergenz} Let $X$ and $W$ have finite second moments and $m \in \mathbb N$.
Then we have $\sum_{k \in \mathbb Z} x_k^{2m}<\infty$ and $\sum_{k \in \mathbb Z}
x_k^{2m}w_k<\infty$.
If additionally $X \in \mathcal G_{\eta}^{2m}$ und $\beta \in L_2([0,1])$, we have
\begin{align*} 
\mathrm E\left|\sum_{k \in \mathbb Z} \langle \beta,\phi_k \rangle \langle \phi_k, X
\rangle\right|^{2m} < \infty.
\end{align*}
\end{lemma}

\begin{lemma} 
\label{Dikscharf}
Let $p \in \mathbb N$ be fixed and suppose $\left( X,W \right)\in \mathcal F_{\eta}^{8p}$ and $\mathrm E|U|^{8p} \leq\eta<\infty$. Then, there is a positive Konstant $C=C_p$ such that for $k \in \mathbb Z$, we have
\begin{align}
\mathrm E | 
I\{ \hat \lambda_k \geq \alpha \gamma_k^\nu \}
\left(D_{i,k,n} - D_{i,k}\right)|^p 
\leq 
\frac{C}{n^{p/2}} 
\left(\frac{w_k^p x_k^{p/2}}{|c_k|^{2p}}+\frac{1}{x_k^{p/2}}\right)
\left( 1 + o(1) \right)  \label{DikDikn}
\end{align}
and 
\begin{align}
\label{Dikn}
\mathrm E \left|I\{ \hat \lambda_k \geq \alpha \gamma_k^\nu \} D_{i,k,n}  \right|^p 
\leq
C_p \Bigg\{ 
\frac{w_k^{p/2}}{|c_k|^p} 
+
\frac{1}{x_k^{p/2}}
+
\frac{C}{n^{p/2}} 
\left(\frac{w_k^p x_k^{p/2}}{|c_k|^{2p}}+\frac{1}{x_k^{p/2}}\right)
\left( 1 + o(1) \right)\Bigg\}.
\end{align}
\end{lemma}

\section{Proof of Theorem \ref{CLTRn3}}
The proof follows by using a central limit theorem for martingal difference sequences with respect to $\left( \mathcal F_{n,j} \right)_{n \in \mathbb N, 0
\leq j \leq n}$, where $
\mathcal F_{n,j}
= \sigma \left( X_1,W_1,Y_1,\ldots,X_j,W_j,Y_j \right)$ and $\mathcal F_{n,0} = \sigma
\left( \emptyset, \Omega \right)$, see \cite{HH80}, Theorem 3.2 and Corollary 3.1, for
\begin{align}
\frac{n}{2t_n} R_{n,3} 
&= 
\sum_{j=2}^n \frac{1}{t_{n} n} \sum_{k \in \mathbb Z}
\mathscr U_{j,k} D_{j,k}
\sum_{i=1}^{j-1}
\ov{\mathscr U_{i,k} D_{i,k}}
 x_k I\{ \lambda_k \geq \alpha \gamma_k^\nu \} \nn =\sum_{j=2}^n Y_{n,j},
\label{Rn12} 
\end{align}
where
\[ Y_{n,j} =  
\frac{1}{t_{n} n} \sum_{k \in \mathbb Z}
\mathscr U_{j,k} 
D_{j,k}
Z_{n,j,k},\] and
\[ Z_{n,j,k} 
=
\sum_{i=1}^{j-1}
\ov{\mathscr U_{i,k} D_{i,k}}
 x_k I\{ \lambda_k \geq \alpha \gamma_k^\nu \} 
.\]
%%%%%%%%%%%%%%%%%%%%%%%%%%%%%%%%%%%%%%%%%%%%%%%%%%%%%%%%%%%%%%%%%%%%%%%%%%%%%%%%%%%%%%%%%%%%%%%%%%

In a first step, we consider the conditional variance of the maringale difference scheme.
\begin{prop}\label{condvar} Under the assumptions of Theorem \ref{CLTteststatistik}, under the null hypothesis and for $(X,W) \in \mathcal F_{\eta}^4$, we have
\[
\mathfrak V_n := \sum_{j=2}^n  
\mathrm E \left[ 
Y_{n,j}^2
\mid \mathcal F_{n,j-1} \right] \Pkonv \mathfrak V \quad \text{as} \quad n \to \infty.
\]
\label{varkonv}
\end{prop}

\begin{proof}
Using that $\mathscr
U_{j,k} D_{j,k}\ov{\mathscr U_{j,l} D_{j,l} }$ is independent of $(\mathcal F_{n,j-1})_{j=1,\ldots,n}$, we can decompose
\begin{align*}
\mathfrak V_n 
&=
\frac{1}{t_n^2 n^2}\sum_{j=2}^n  
\e \Big[ 
\Big|\sum_{k \in  \mathbb Z} 
\mathscr U_{j,k} D_{j,k} Z_{n,j,k} \Big|^2 \mid \mathcal F_{n,j-1} \Big] \\
%%%%%%%%%%%%%%%%%%%%%%%%%%%%%%%%%%%%%%%%%%%%%%%%%%%%%%%%%%%%%%%%%%%%%%%%%%%%%%%%%%%
&=
\frac{1}{t_n^2 n} 
\sum_{k \in \mathbb Z}
x_k \left( \frac{ x_k w_k}{|c_k|^2} - 1 \right) I\{ \lambda_k \geq \alpha \gamma_k^\nu \}
\e |\mathscr U_{1,k}|^2 \\
&\phantom{=}
\phantom{\frac{1}{t_n^2 n^2} \sum_{j=2}^n \sum_{k \in \mathbb Z}}
\Bigg(
\sum_{i=1}^{n-1}
|\mathscr U_{i,k} D_{i,k}|^2 
+
\sum_{\substack{i,p=1,\\ i \neq p}}^{n-1}
\mathscr U_{i,k} D_{i,k}
\ov{\mathscr U_{p,k} D_{p,k}} \Bigg) \\
%%%%%%%%%%%%%%%%%%%%%%%%%%%%%%%%%%%%%%%%%%%%%%%%%%%%%%%%%%%%%%%%%%%%%%%%%%%%%%%%%%%
&= \mathfrak V_{n,1} + \mathfrak V_{n,2}.
\end{align*}
We define
\begin{align*}
\mathfrak H_n
:= 
\frac{\mathfrak V}{t_n^2 n} 
\sum_{k \in \mathbb Z} 
x_k \left( \frac{ x_k w_k}{|c_k|^2} - 1 \right)
 I\{ \lambda_k \geq \alpha \gamma_k^\nu \}
\sum_{i=1}^{n-1}  
\e|D_{i,k}|^2
\end{align*} and show
\begin{align*}
\mathfrak V_{n,1} = \mathfrak H_n + o(1)
\end{align*}
by proving the corresponding $L_2$-convergence. Afterwards we show that $\mathfrak H_n$ converges in probability to $\mathfrak V$. 
Writing for $i \in \{1,\ldots,n\}$ and $k \in \mathbb Z$ 
\begin{align*}
&|\mathscr U_{i,k} D_{i,k}|^2
\e |\mathscr U_{1,k}|^2 
-
\mathfrak V
\e|D_{i,k}|^2 \\
&=
\mathfrak V^{1/2}
\Big[  
|\mathscr U_{i,k} D_{i,k}|^2 
-
\mathfrak V^{1/2}
\e|D_{i,k}|^2
\Big] 
-
|\mathscr U_{i,k} D_{i,k}|^2 |\langle \beta, \phi_k \rangle|^2 x_k.
\end{align*}
and, observing that $\sigma^2 + \sum_{m \in \mathbb Z} |\langle \beta, \phi_m \rangle|^2 x_m  \leq C_1$ for some constant $C_1>0$, we get
\begin{align*}
&\e \left( \mathfrak V_{n,1} - \mathfrak H_n \right)^2 \leq \mathbb V_{n,1} + \mathbb V_{n,2} + \mathbb V_{n,3}
\end{align*}
%%%%%%%%%%%%%%%%%%%%%%%%%%%%%%%%%%%%%%%%%%%%%%%%%%%%%%%%%%%%%%%%%%%%%%%%%%%%%%%%%%%%%
with
\begin{align*}
\mathbb V_{n,1}&=
\frac{C}{t_n^4 n^2}
\sum_{k \in \mathbb Z} 
x_k^2 \left( \frac{ x_k w_k}{|c_k|^2} - 1 \right)^2
 I\{ \lambda_k \geq \alpha \gamma_k^\nu \} \\
&\phantom{=}
\phantom{\frac{C}{t_n^4 n^2} \sum_{k \in \mathbb Z}}
\Bigg\{
\sum_{i=1}^{n-1}
\e \Big(  
|\mathscr U_{i,k} D_{i,k}|^2 
 -  
\mathfrak V^{1/2}
\e|D_{i,k}|^2 \Big)^2\\
&\phantom{=}
\phantom{\frac{C}{t_n^4 n^2} \sum_{k \in \mathbb Z}\Big\{}
+
\sum_{\substack{i,p=1, \\ i \neq p}}^{n-1}
\e \Big[  
|\mathscr U_{i,k} D_{i,k}|^2 -  
\mathfrak V^{1/2}
\e|D_{1,k}|^2 \Big]
\e \Big[ (|\mathscr U_{p,k} D_{p,k}|^2 -  
\mathfrak V^{1/2}
\e|D_{1,k}|^2 \Big] 
\Bigg\}\\
%%%%%%%%%%%%%%%%%%%%%%%%%%%%%%%%%%%%%%%%%%%%%%%%%%%%%%%%%%%%%%%%%%%%%%%%%%%%%%%%%%%%%
\mathbb V_{n,2}&{=}
%+
\frac{C}{t_n^4 n^2}
\sum_{\substack{k,l \in \mathbb Z, \\ |k| \neq |l|}} 
x_k \left( \frac{ x_k w_k}{|c_k|^2} - 1 \right)
 I\{ \lambda_k \geq \alpha \gamma_k^\nu \} 
x_l \left( \frac{ x_l w_l}{|c_l|^2} - 1 \right)
 I\{ \lambda_l \geq \alpha \gamma_l^\nu \}  \\
&\phantom{=}
\phantom{+\frac{1}{t_n^4 n^2}} 
\phantom{\sum_{\substack{k,l \in \mathbb Z, \\ k \neq l}} }
\Bigg\{
\sum_{i=1}^{n-1}
\e \Big[ 
\Big( |\mathscr U_{i,k} D_{i,k}|^2  
-  
\mathfrak V^{1/2}
\e|D_{i,k}|^2\Big)
\Big( |\mathscr U_{i,l} D_{i,l}|^2 
-  
\mathfrak V^{1/2}
\e|D_{i,l}|^2\Big)
\Big] \\
&\phantom{=}
\phantom{+\frac{1}{t_n^4 n^2}} 
\phantom{\sum_{\substack{k,l \in \mathbb Z, \\ k \neq l}} }
+\sum_{\substack{i,p=1, \\ i \neq p}}^{n-1}
\e \Big[ 
|\mathscr U_{i,k} D_{i,k}|^2  
-  
\mathfrak V^{1/2}
\e|D_{i,k}|^2\Big] 
\e \Big[ |\mathscr U_{i,l} D_{i,l}|^2 
-  
\mathfrak V^{1/2}
\e|D_{i,l}|^2\Big] 
\Bigg\} \\
%%%%%%%%%%%%%%%%%%%%%%%%%%%%%%%%%%%%%%%%%%%%%%%%%%%%%%%%%%%%%%%%%%%%%%%%%%%%%%%%%%%%%%%
\mathbb V_{n,3}&{=}
%+
\frac{2}{t_n^4 n^2}
\e \Bigg(  
\sum_{k \in \mathbb Z} 
x_k^2 \left( \frac{ x_k w_k}{|c_k|^2} - 1 \right)
|\langle \beta, \phi_k \rangle|^2
I\{ \lambda_k \geq \alpha \gamma_k^\nu \}
\sum_{i=1}^{n-1}
|\mathscr U_{i,k} D_{i,k}|^2  
\Bigg)^2. \\
\end{align*}
We have
\begin{align}
\e |\mathscr U_{j,k} D_{j,k} |^2 
=
\Bigg( \sigma^2 + \sum_{\substack{m \in \mathbb Z,\\ |m| \neq |k|}} 
|\langle \beta, \phi_m \rangle|^2 x_m \Bigg) 
\left( \frac{w_k}{|c_k|^2} - \frac{1}{x_k} \right), \label{UDquadkorrekt}
\end{align}
because $|\mathscr U_{j,k}|^2$ and $|D_{j,k}|^2$ are uncorrelated for all $k \in \mathbb Z$ and $j \in \{1,\ldots,n\}$. With Lemma
\ref{Reihenkonvergenz} and \eqref{Dik4}, for all $i \in \{1,\ldots,n\}$ and $k \in
\mathcal K_n$, we have
\begin{align}
\e \left( |\mathscr U_{i,k} D_{i,k}|^2 
 - 
\mathfrak V^{1/2}
\e|D_{i,k}|^2 \right)^2 
&\leq
C \left( \e |D_{1,k}|^4 - \left( \e |D_{1,k}|^2 \right)^2 \right) \nn \leq
C \e |D_{1,k}|^4 \\
&\leq
\frac{C}{\alpha^2}. \label{UDquaddiff}
\end{align}
as well as
\begin{align}
\e \left[ |\mathscr U_{i,k} D_{i,k}|^2
 -
\mathfrak V^{1/2}
\e|D_{i,k}|^2 \right] 
&=
- \e|D_{1,k}|^2 |\langle \beta, \phi_k \rangle|^2 x_k \nn \\
=
-\left( \frac{x_k w_k}{|c_k|^2} - 1 \right)|\langle \beta, \phi_k \rangle|^2.
\label{UDdiffvereinfacht}
\end{align}
For the mixed terms with $k,l \in \mathbb Z, |k| \neq |l|$ and $i \in \{1,\ldots,n\}$ and $\frac{w_k}{|c_k|^2} - \frac{1}{x_k} \geq 0$ for
all $k \in \mathbb Z$, we get
\begin{align}
&\e \Big[ 
\Big( |\mathscr U_{i,k} D_{i,k}|^2  
-  
\mathfrak V^{1/2}
\e|D_{i,k}|^2\Big) 
\Big( |\mathscr U_{i,l} D_{i,l}|^2 
-  
\mathfrak V^{1/2}
\e|D_{i,l}|^2\Big)
\Big] \nn \\
&\leq
\e \Big[  
| \mathscr U_{1,k} D_{1,k} \mathscr U_{1,l} D_{1,l}|^2
\Big]
+
\Big( \frac{w_k}{|c_k|^2} - \frac{1}{x_k} \Big)
\Big( \frac{w_l}{|c_l|^2} - \frac{1}{x_l} \Big) \nn \\
&\leq
C\Bigg\{\frac{1}{\alpha^2}
| \langle \beta, \phi_k \rangle |^2 x_k | \langle \beta, \phi_l \rangle |^2 x_l
+
\frac{x_l}{\alpha}
| \langle \beta, \phi_l \rangle |^2 \Big( \frac{w_k}{|c_k|^2} - \frac{1}{x_k} \Big) \nn \\
&\phantom{=}
\phantom{C\Bigg\{}
+
\frac{x_k}{\alpha}
| \langle \beta, \phi_k \rangle |^2 \Big( \frac{w_l}{|c_l|^2} - \frac{1}{x_l} \Big)
+
\Big( \frac{w_k}{|c_k|^2} - \frac{1}{x_k} \Big)
\Big( \frac{w_l}{|c_l|^2} - \frac{1}{x_l} \Big) \Bigg\}.  \label{UDgemdiff}
\end{align}
Using this, we have 
\begin{align*}
\mathbb V_{n,1}
&\leq
\frac{C}{t_n^4 n^2}
\sum_{k \in \mathbb Z} 
x_k^2 \left( \frac{ x_k w_k}{|c_k|^2} - 1 \right)^2
 I\{ \lambda_k \geq \alpha \gamma_k^\nu \} 
%\\
%&\phantom{=}
%\phantom{\frac{1}{t_n^4 n^2} \sum_{k \in \mathbb Z} }
\Bigg\{
\frac{n}{\alpha^2}
+
n^2 \left( \frac{x_k w_k}{|c_k|^2} - 1 \right)^2
|\langle \beta, \phi_k \rangle|^4
\Bigg\} \\
&\leq
\frac{C}{t_n^4 n \alpha^2}
\sum_{k \in \mathbb Z} 
x_k^2 \left( \frac{ x_k w_k}{|c_k|^2} - 1 \right)^2
I\{ \lambda_k \geq \alpha \gamma_k^\nu \} 
\\
&\phantom{=}
+
\frac{C}{t_n^4}
\sum_{k \in \mathbb Z} 
x_k^2 \left( \frac{ x_k w_k}{|c_k|^2} - 1 \right)^4
|\langle \beta, \phi_k \rangle|^4
 I\{ \lambda_k \geq \alpha \gamma_k^\nu \}	\\
&= o \left( 1 + \frac{1}{t_n^2} \right),
\end{align*}
with some constant $C>0$. With similar arguments we obtain
\begin{align*}
\mathbb V_{n,2}
&=
o\left( 
1
+
\frac{1}{t_n^2}
+
\frac{1}{\sqrt{n} t_n}
\right)
+ \mathcal O \left( \frac{1}{n} \right).
\end{align*}
and
\begin{align*}
\mathbb V_{n,3} 
%%%%%%%%%%%%%%%%%%%%%%%%%%%%%%%%%%%%%%%%%%%%%%%%%%%%%%%%%%%%%%%%%%%%%%%%%%%%%%%%%%%%%%%%
%%%%%%%%%%%%%%%%%%%%%%%%%%%%%%%%%%%%%%%%%%%%%%%%%%%%%%%%%%%%%%%%%%%%%%%%%%%%%%%%%%%%%%%%
&=
o \left( 1 + \frac{1}{t_n^2} + \frac1n + \frac{1}{\sqrt{n} t_n} \right).
\end{align*}
which altogether results in
\[ \mathfrak V_{n,1} = \mathfrak H_n + o_{P} \left( 1 \right). \]
The stochastic convergence of $\mathfrak H_n$ follows by
\begin{align*}
\mathfrak H_n 
&=
\mathfrak V
\frac{n-1}{t_n^2 n} 
\sum_{k \in \mathbb Z} 
\left( \frac{ x_k w_k}{|c_k|^2} - 1 \right)^2
I\{ \lambda_k \geq \alpha \gamma_k^\nu \} 
\stackrel{P}{\to}\mathfrak V
\end{align*}
for $n\to\infty$.
%
%
%
%
%
%
%
%
%
%
%     ~ Diskusssion von Vn2  ~
%
%
%
%
%
%
%
%
%
%
For proving, that $\mathfrak V_{n,2}$ converges stochastically to 0 we show again the corresponding $L_2$-convergence. To this end, we bound for all $i \in \{1,\ldots,n\}$ und
$k \in \mathbb Z$ the term $\e |\mathscr
U_{1,k}|^2$ by a constant $C<\infty$ using the centeredness of $U$ and Lemma
\ref{Reihenkonvergenz}, to obtain

\[ \mathfrak V_{n,2} = o_P \left( 1 \right). \]
The detailed arguments can be found in the supplementary material.
\end{proof}
%%%%%%%%%%%%%%%%%%%%%%%%%%%%%%%%%%%%%%%%%%%%%%%%%%%%%%%%%%%%%%%%%%%%%%%%%%%%%%%%%%%%%%%

%%%%%%%%%%%%%%%%%%%%%%%%%%%%%%%%%%%%%%%%%%%%%%%%%%%%%%%%%%%%%%%%%%%%%%%%%%%%%%%%%%%%%%%%

The second step is to show the conditional Lindeberg condition by verifying an unconditional Ljapunov condition.

\begin{prop}\label{lindeberg} Under the assumptions of Theorem \ref{CLTteststatistik}, under the null hypothesis, and with $(X,W) \in \mathcal F_{\eta}^4$, we have
\begin{equation}
\forall \, \varepsilon >0: \; \sum_{j=2}^n  
\mathrm E \left[ 
Y_{n,j}^2 I\{ |Y_{n,j}| > \varepsilon \}
\mid \mathcal F_{n,j-1} \right] \Pkonv 0 \quad \text{as} \quad n \to \infty.
\end{equation}
\label{Lindeberg}
\end{prop}

%-----------------------------------------------------------------------------------
%
\bigskip
\begin{proof}
It is shown in \cite{AAM14} and \cite{GSS78} that the conditional Lindeberg condition
follows from the unconditional Ljapunov condition. We will show in the following, that \[ \sum_{j=2}^n \e|Y_{n,j}|^4 = o(1) \]
and decompose
\[\sum_{j=2}^n \e|Y_{n,j}|^4=L_{n,1}+L_{n,2}+L_{n,3}+L_{n,4},\]
where
\begin{align*}
L_{n,1}
&=
\frac{1}{t_n^4 n^4} \sum_{j=2}^n 
\sum_{k \in \mathbb Z} 
\e \left|\mathscr U_{j,k} D_{j,k} Z_{n,j,k}\right|^4,\\
L_{n,2}&=
\frac{1}{t_n^4 n^4} \sum_{j=2}^n\sum_{\substack{k,l \in \mathbb Z,\\ |k| \neq |l|}} 
\e \left|\mathscr U_{j,k} D_{j,k} Z_{n,j,k} \ov{\mathscr U_{j,l} D_{j,l}
Z_{n,j,l}}\right|^2, \\
L_{n,3}&{=}
{\frac{1}{t_n^4 n^4} \sum_{j=2}^n }
\sum_{\substack{k,l,q \in \mathbb Z,\\ |k|,|l| \neq |q|, |k| \neq |l|}} 
\e \Big[|\mathscr U_{j,k} D_{j,k} Z_{n,j,k}|^2 
\mathscr U_{j,l} D_{j,l} Z_{n,j,l}
\ov{\mathscr U_{j,q} D_{j,q} Z_{n,j,q}}\Big], \\
L_{n,4}&{=}
{\frac{1}{t_n^4 n^4} \sum_{j=2}^n }
\sum_{\substack{k,l,p,q \in \mathbb Z,\\ |k|,|l|,|p| \neq |q|, \\ |k|,|l| \neq |p|, |k|
\neq |l|}} \e \Big[\mathscr U_{j,k} D_{j,k} Z_{n,j,k} 
\ov{\mathscr U_{j,l} D_{j,l} Z_{n,j,l}}
\mathscr U_{j,p} D_{j,p} Z_{n,j,p}
\ov{\mathscr U_{j,q} D_{j,q} Z_{n,j,q}}\Big].
\end{align*}
%
%
%
%
%
%
% ------- Ln1 ---------
%
%
%
%
%
%
%
For $L_{n,1}$ we use, that for all $k \in \mathbb Z, n \in
\mathbb N, j \in \{1,\ldots,n\}$, $Z_{n,j,k}$ are stochastically independent of $\mathscr
U_{j,k} D_{j,k}$ and $\mathscr U_{j,k}$ are uncorrelated with $D_{j,k}$. Furthermore, the
fourth absolute moment of $\mathscr U_{j,k}$ is due to the centredness of $U$ and Lemma
\ref{Reihenkonvergenz} uniformly bounded. The fourth absolute moment of $D_{j,k}$ can be
estimated using Assumption \ref{ass:regulari} and $(X,W) \in \mathcal F_{\eta}^4$ as
\begin{align}
\e |D_{j,k}|^4 
\leq 
C \left( \frac{\e |\langle W, \phi_k \rangle |^4}{|c_k|^4} 
+
\frac{\e |\langle X, \phi_k \rangle |^4}{x_k^4} \right)
\leq
C \eta \left( \frac{w_k^2}{|c_k|^4} 
+
\frac{1}{x_k^2} \right)
\leq
\frac{C \eta}{\alpha^2}. \label{Dik4}
\end{align} 
Again using similar arguments, we obtain
\begin{align}
\e \left| 
\mathscr U_{i_1,k} D_{i_1,k}\right|^2 
=
\e | \mathscr U_{i_1,k} |^2 \e |D_{i_1,k}|^2 
&\leq
C \left( \frac{w_k}{|c_k|^2} - \frac{1}{x_k} \right). \label{UDquad}
\end{align}
This results in
\begin{align}
&\e \Big|
\sum_{i=1}^{j-1}
\mathscr U_{i,k}
D_{i,k}
x_k I\{ \lambda_k \geq \alpha \gamma_k^\nu \} 
\Big|^4 \nn \\
%%%%%%%%%%%%%%%%%%%%%%%%%%%%%%%%%%%%%%%%%%%%%%%%%%%%%%%%%%%%%%%%%%%%%%%%%%%%%%%%%%%%%
&=
x_k^4 I\{ \lambda_k \geq \alpha \gamma_k^\nu \} 
\Bigg\{
\sum_{i=1}^{j-1}
\e |\mathscr U_{i,k}|^4
\e |D_{i,k}|^4
+
2 \sum_{1 \leq i_1 < i_2 \leq j-1}
\e |\mathscr U_{i_1,k} D_{i_1,k}|^2
\e |\mathscr U_{i_2,k} D_{i_2,k}|^2
\Bigg\} \nn \\
&\leq
\frac{Cn}{\alpha^2} x_k^4 I\{ \lambda_k \geq \alpha \gamma_k^\nu \} 
+
C n^2 x_k^2 
\left( \frac{x_k w_k}{|c_k|^2} - 1 \right)^2
I\{ \lambda_k \geq \alpha \gamma_k^\nu \}. \label{Znjk4}
\end{align}
Putting these results together, we get
\begin{align*}
L_{n,1}
&=
\frac{1}{t_n^4 n^4} \sum_{j=2}^n 
\sum_{k \in \mathbb Z} 
\e |\mathscr U_{j,k}|^4 
\e |D_{j,k}|^4 
\e |Z_{n,j,k}|^4 \\
&\leq
\frac{C}{t_n^4 n^4 \alpha^2} \sum_{j=2}^n 
\sum_{k \in \mathbb Z} 
\e \Big|
\sum_{i=1}^{j-1}
\mathscr U_{i,k}
D_{i,k}
x_k I\{ \lambda_k \geq \alpha \gamma_k^\nu \} 
\Big|^4 \\
&\leq
\frac{C}{t_n^4 n \alpha^2}
\sum_{k \in \mathbb Z} 
x_k^2
I\{ \lambda_k \geq \alpha \gamma_k^\nu \} 
\left(
\frac{1}{n \alpha^2} x_k^2  
+
\left( \frac{x_k w_k}{|c_k|^2} - 1 \right)^2  
\right) \\
&=
o(1) \frac{1}{t_n^4} 
\left( \sum_{k \in \mathbb Z} 
 x_k^4
I\{ \lambda_k \geq \alpha \gamma_k^\nu \} 
+
\sum_{k \in \mathbb Z} 
 x_k^2
\left( \frac{x_k w_k}{|c_k|^2} - 1 \right)^2
I\{ \lambda_k \geq \alpha \gamma_k^\nu \}\right),
\end{align*}
where the first series converges due to Lemma \ref{Reihenkonvergenz} and the second series either also converges or, if not, can be bounded by $C t_n^2$. \\

%
%
%
%
%
%
% ------- Ln4 ---------
%
%
%
%
%
%
% 
Considering $L_{n,4}$, we use the stochastic independence of $Z_{n,j,k}$ and $\mathscr U_{j,l}
D_{j,l}$ for all $k,l \in \mathbb Z$, which results in
\begin{align*}
&\e \big[\mathscr U_{j,k} D_{j,k} Z_{n,j,k} 
\ov{\mathscr U_{j,l}} \ov{D_{j,l}} \ov{Z_{n,j,l}}
\mathscr U_{j,p} D_{j,p} Z_{n,j,p}
\ov{\mathscr U_{j,q}} \ov{D_{j,q}} \ov{Z_{n,j,q}}\big] \\
&=
\e \big[\mathscr U_{j,k} D_{j,k}  
\ov{\mathscr U_{j,l}} \ov{D_{j,l}} 
\mathscr U_{j,p} D_{j,p} 
\ov{\mathscr U_{j,q}} \ov{D_{j,q}} \big]
\e \big[ Z_{n,j,k} \ov{Z_{n,j,l}} Z_{n,j,p} \ov{Z_{n,j,q}}\big].
\end{align*}
The rest of the argumentation is just calculating the expectations using that
for all $j \in \{1,\ldots,n\}$, 
$D_{j,k}, D_{j,l}, D_{j,p}$ and $D_{j,q}$ are uncorellated with $S_{j,m}$ for all $m \in
\mathbb Z \backslash \{m \in \mathbb Z: |m| = |k|,|l|,|p|,|q|\}$ and stochastically independent of $U_j$. Finally,
\begin{align}
\e [S_{j,k} D_{j,k}]
&=
\langle \beta, \phi_k \rangle
\e \left[
\langle \phi_k, X_j \rangle 
\left( 
\frac{\langle W_j, \phi_k \rangle}{c_k} - \frac{\langle X_j, \phi_k \rangle}{x_k}
\right)\right] =
\langle \beta, \phi_k \rangle
\left( \frac{c_k}{c_k} - \frac{x_k}{x_k} \right)=0 \label{SD}
\end{align}
and, in the same way, $\e [\ov{S_{j,k}}D_{j,k}]=\e [S_{j,k} \ov{D_{j,k}}]=0$, which gives $L_{n,4}=0$.
%
%
%
%
%
%
% ------- Ln2 ---------
%
%
%
%
%
%
% 
\bigskip

With similar arguments as above, which can be found in the supplementary material we get

\begin{align*}
L_{n,2} 
&=
o\left(\frac{1}{t_n^4} + \frac{1}{t_n^2 n} 
+ \frac{1}{t_n^2}
+  \frac{1}{t_n \sqrt{n}}
\right)
+
\mathcal O \left(\frac1n +\frac{1}{n^2} \right) =
o(1),
\end{align*}
and
%
%
%
%
%
%
%
% ------- Ln3 ---------
%
%
%
%
%
%
% 
\bigskip

\begin{align*}
L_{n,3}
= o \left( \frac{1}{t_n^2 n} \right).
\end{align*} 
\end{proof}

All remainder terms can be estimated with similar techniques. We exemplarily show the idea for Proposition \ref{Rn1245}, that is for $R_{n,2}$, in the supplementary material.

%%%%%%%%%%%%%%%%%%%%%%%%%%%

\section{Proof of Theorem \ref{CLTboot}}

Let $\Phi_{\mathfrak V}(\cdot)$ denote the distribution function of the normal distribution with mean zero and variance $\mathfrak V$, $F_n$ the distribution function of $\frac{n}{t_n}
\left(T_n -\mathfrak B_n- \mathfrak R_n \right)$ and $F_{\mathcal S_n,n}^*$ the distribution function of the conditional distribution of $\frac{n}{t_n} \left(T_n^* -\mathfrak B_n^*- \mathfrak R_n^* \right)$ given
$\mathcal S_n$.
By bounding
\begin{align*}
\sup_{t \in \mathbb R}
\left| F_{\mathcal S_n,n}^*(t) - F_n(t) \right|
&\leq
\sup_{t \in \mathbb R}
\left| F_{\mathcal S_n,n}^*(t) - \Phi_{\mathfrak V}(t) \right|
+
\sup_{t \in \mathbb R}
\left| F_n(t) - \Phi_{\mathfrak V}(t)\right| =: M_{1,n} + M_{2,n},
\end{align*}
similar to the example in Section 29 of \cite{Das08}, it is enough to show the convergence
of $M_{1,n}$ and $M_{2,n}$. Due to the continuity of $\phi_{\mathfrak V}$, the convergence of $M_{2,n}$ follows directly from Theorem \ref{CLTteststatistik} and Polya's Theorem, as stated in Section 1.5.3 of \cite{Ser80}. 
Again, using Polya's Theorem, it is enough to show for $M_{1,n}$, that
for all $\varepsilon >0$
\begin{align}
\lim_{n \to \infty} \mathbb P \left( 
\left|F_{\mathcal S_n,n}^*(t) - \phi_{\mathfrak V}(t)
\right| > \varepsilon \right)
= 0.
\label{polya}
\end{align}
For this we just immitate the proof of Theorem \ref{CLTteststatistik}. Analogously to (\ref{ZerlTeststat}), we decompose 
\begin{align*}
\frac{n}{t_n} T_n^* 
&=
\frac{1}{t_n} \sum_{j=1}^n 
\left| \left\langle 
T_{n,1}^* + T_{n,2}^* + T_{n,3}^*, X_j
\right\rangle \right|^2
+
\frac{1}{t_n} \sum_{j=1}^n 
\langle 
T_{n,1}^* + T_{n,2}^*+ T_{n,3}^*, X_j \rangle
\langle X_j, R_n^* \rangle \\
&\phantom{=}
+
\frac{1}{t_n} \sum_{j=1}^n 
\langle 
X_j, T_{n,1}^* + T_{n,2}^* + T_{n,3}^* \rangle
\langle R_n^*, X_j \rangle
+
\frac{1}{t_n} \sum_{j=1}^n 
\left| \left\langle 
R_n^*, X_j
\right\rangle \right|^2,
\end{align*}
where, similar to the proof of Theorem \ref{CLTteststatistik}, we get
\[ \frac{1}{t_n} \sum_{j=1}^n 
\left| \left\langle 
R_n^*, X_j
\right\rangle \right|^2 -\frac{n}{t_n}\mathfrak R_n 
=
\frac{n}{t_n} R_{n,3}^* + \frac{n}{t_n} \left( R_{n,2}^* - \mathfrak R_n \right)
+ \frac{n}{t_n} \left( R_{n,1}^* + R_{n,4}^* + R_{n,5}^* \right).  \] 
Then, $\frac{n}{t_n} (R_{n,3}^*-\mathfrak B_n^*-\mathfrak R_n^*)$ converges weakly in probability
to $\mathcal N(0, \mathfrak V)$ along the lines of Theorem \ref{CLTRn3}.
The remainder terms can be discussed to be negligible with the same arguments as for the remainder terms in Theorem \ref{CLTteststatistik}.
\hfill $\Box$

\section{Supplementary Material}
\subsection{Proof of Proposition A.5}
%, $R_{n,2}$}
We give only the proof for $R_{n,2}$. We have
\begin{align*}
&\frac{n^2}{t_n^2}\e | R_{n,2} - \mathfrak R_n|^2 \\
%%%%%%%%%%%%%%%%%%%%%%%%%%%%%%%%%%%%%%%%%%%%%%%%%%%%%%%%%%%%%%%%%%%%%%%%%%%%%%%%%%%%%%%
%%%%%%%%%%%%%%%%%%%%%%%%%%%%%%%%%%%%%%%%%%%%%%%%%%%%%%%%%%%%%%%%%%%%%%%%%%%%%%%%%%%%%%%
&\leq
\frac{1}{t_n^2 n^2} \sum_{k \in \mathbb Z} 
x_k^2 I\{ \lambda_k \geq \alpha \gamma_k^\nu \} 
\e \Bigg|
\sum_{i=1}^n
\left(|D_{i,k} \mathscr U_{i,k}|^2  
-
\mathfrak V^{1/2}
\left(\frac{w_k}{|c_k|^2} -\frac{1}{x_k} \right)\right)
\Bigg|^2 \\
%%%%%%%%%%%%%%%%%%%%%%%%%%%%%%%%%%%%%%%%%%%%%%%%%%%%%%%%%%%%%%%%%%%%%%%%%%%%%%%%%%%%%%%
&\phantom{=}
+
\frac{1}{t_n^2 n^2} \sum_{\substack{k,l \in \mathbb Z, \\ |k| \neq |l|}} 
x_k I\{ \lambda_k \geq \alpha \gamma_k^\nu \} 
x_l I\{ \lambda_l \geq \alpha \gamma_l^\nu \} \\
&\phantom{=}
\phantom{\frac{1}{t_n^2 n^2} \sum_{\substack{k,l \in \mathbb Z, \\ k \neq l}} }
\sum_{i=1}^n
\e \Big[
\Big(|D_{i,k} \mathscr U_{i,k}|^2  
-
\mathfrak V^{1/2}
\Big(\frac{w_k}{|c_k|^2} -\frac{1}{x_k} \Big)\Big) 
\Big(|D_{i,l} \mathscr U_{i,l}|^2  
-
\mathfrak V^{1/2}
\Big(\frac{w_l}{|c_l|^2} -\frac{1}{x_l} \Big)\Big)
\Big] \\
%%%%%%%%%%%%%%%%%%%%%%%%%%%%%%%%%%%%%%%%%%%%%%%%%%%%%%%%%%%%%%%%%%%%%%%%%%%%%%%%%%%%%%%
&\phantom{=}
+
\frac{1}{t_n^2 n^2} \sum_{\substack{k,l \in \mathbb Z, \\ |k| \neq |l|}} 
x_k I\{ \lambda_k \geq \alpha \gamma_k^\nu \} 
x_l I\{ \lambda_l \geq \alpha \gamma_l^\nu \} \\
&\phantom{=}
\phantom{\frac{1}{t_n^2 n^2} \sum_{\substack{k,l \in \mathbb Z, \\ k \neq l}} }
\sum_{\substack{i,p=1,\\ i \neq p}}^n
\e \Big[
|D_{i,k} \mathscr U_{i,k}|^2  
-
\mathfrak V^{1/2}
\Big(\frac{w_k}{|c_k|^2} -\frac{1}{x_k} \Big)\Big]
\e \Big[|D_{p,l} \mathscr U_{p,l}|^2  
-
\mathfrak V^{1/2}
\Big(\frac{w_l}{|c_l|^2} -\frac{1}{x_l} \Big)\Big].
\end{align*}
The terms quadratic in $k \in \mathbb Z$ can be estimated by
Lemma B.1 und \eqref{Dik4}, while the other terms except the one coming from $|\langle \beta, \phi_k \rangle |^2 x_k$ vanish
\begin{align*}
\e \Bigg|
\sum_{i=1}^n
\left(|D_{i,k} \mathscr U_{i,k}|^2  
-
\mathfrak V^{1/2} 
\left(\frac{w_k}{|c_k|^2} -\frac{1}{x_k} \right)\right)
\Bigg|^2\leq
\frac{Cn}{\alpha^2}
+
Cn^2 \left(\frac{x_k w_k}{|c_k|^2} -1 \right)^2
|\langle \beta, \phi_k \rangle |^4 \left(1 + \frac1n\right).
\end{align*}
Using the Cauchy-Schwarz inequality \eqref{UDquaddiff}, leads to
\begin{align*}
&\e \Bigg[
\left(|D_{i,k} \mathscr U_{i,k}|^2  
-
\mathfrak V^{1/2}
\left(\frac{w_k}{|c_k|^2} -\frac{1}{x_k} \right)\right) 
\left(|D_{i,l} \mathscr U_{i,l}|^2  
-
\mathfrak V^{1/2}
\left(\frac{w_l}{|c_l|^2} -\frac{1}{x_l} \right)\right)
\Bigg] \leq 
\frac{C}{\alpha^2}.
\end{align*}
The expectations with $k,l \in \mathbb Z$, $|k| \neq |l|$ und $i,p \in \{1,\ldots,n\}$,
$i \neq p$ can be estimated by \eqref{UDdiffvereinfacht}. This
finally yields
\begin{align*}
&\frac{n^2}{t_n^2}\e | R_{n,2} - \mathfrak R_n|^2 \\
%%%%%%%%%%%%%%%%%%%%%%%%%%%%%%%%%%%%%%%%%%%%%%%%%%%%%%%%%%%%%%%%%%%%%%%%%%%%%%%%%%%%%%%
&\leq
\frac{1}{t_n^2 n^2} \sum_{k \in \mathbb Z} 
x_k^2 I\{ \lambda_k \geq \alpha \gamma_k^\nu \} 
\Bigg\{
\frac{Cn}{\alpha^2}
+
Cn^2 \left(\frac{x_k w_k}{|c_k|^2} -1 \right)^2
|\langle \beta, \phi_k \rangle |^4 \left(1 + \frac1n\right)  \Bigg\} \\
%%%%%%%%%%%%%%%%%%%%%%%%%%%%%%%%%%%%%%%%%%%%%%%%%%%%%%%%%%%%%%%%%%%%%%%%%%%%%%%%%%%%%%%
&\phantom{=}
+
\frac{C}{t_n^2 n^2} \sum_{\substack{k,l \in \mathbb Z, \\ |k| \neq |l|}} 
x_k I\{ \lambda_k \geq \alpha \gamma_k^\nu \} 
x_l I\{ \lambda_l \geq \alpha \gamma_l^\nu \} \\
&\phantom{=}
\phantom{\frac{1}{t_n^2 n^2} \sum_{\substack{k,l \in \mathbb Z, \\ k \neq l}} }
\Bigg\{ 
\frac{n }{\alpha^2} 
+
n (n -1) \left( \frac{x_k w_k}{|c_k|^2} - 1 \right)|\langle \beta, \phi_k \rangle |^2 
\left( \frac{x_l w_l}{|c_l|^2} - 1 \right)|\langle \beta, \phi_l \rangle |^2
\Bigg\} \\
&=
o \left( 1 + \frac{1}{t_n^2} \right).
\end{align*}
The second part can be shown by using
\begin{align}
\frac{x_k w_k}{|c_k|^2} - 1  \leq \frac{1}{\alpha} (x_k - \lambda_k),
\label{difftn}
\end{align} for all $k \in \mathscr K_n$ together with Lemma B.1 
\hfill $\Box$

All the other parts of Proposition A.5 as well as Lemmas A.2-A.4 follow by very similar techniques. For details we refer to [5] in the main article.

\subsection{Details for the proof of Proposition C.1}
%%%%%%%%%%%%%%%%%%%%%%%%%%%%%%%%%%

Using that $\mathscr
U_{j,k} D_{j,k}\ov{\mathscr U_{j,l} D_{j,l} }$ is independent of $(\mathcal F_{n,j-1})_{j=1,\ldots,n}$, we can decompose
\begin{align*}
\mathfrak V_n 
&=
\frac{1}{t_n^2 n^2}\sum_{j=2}^n  
\e \Big[ 
\Big|\sum_{k \in  \mathbb Z} 
\mathscr U_{j,k} D_{j,k} Z_{n,j,k} \Big|^2 \mid \mathcal F_{n,j-1} \Big] \\
%%%%%%%%%%%%%%%%%%%%%%%%%%%%%%%%%%%%%%%%%%%%%%%%%%%%%%%%%%%%%%%%%%%%%%%%%%%%%%%%%%%
&=
\frac{1}{t_n^2 n} 
\sum_{k \in \mathbb Z}
x_k \left( \frac{ x_k w_k}{|c_k|^2} - 1 \right) I\{ \lambda_k \geq \alpha \gamma_k^\nu \}
\e |\mathscr U_{1,k}|^2 \\
&\phantom{=}
\phantom{\frac{1}{t_n^2 n^2} \sum_{j=2}^n \sum_{k \in \mathbb Z}}
\Bigg(
\sum_{i=1}^{n-1}
|\mathscr U_{i,k} D_{i,k}|^2 
+
\sum_{\substack{i,p=1,\\ i \neq p}}^{n-1}
\mathscr U_{i,k} D_{i,k}
\ov{\mathscr U_{p,k} D_{p,k}} \Bigg) \\
%%%%%%%%%%%%%%%%%%%%%%%%%%%%%%%%%%%%%%%%%%%%%%%%%%%%%%%%%%%%%%%%%%%%%%%%%%%%%%%%%%%
&= \mathfrak V_{n,1} + \mathfrak V_{n,2}.
\end{align*}
We define
\begin{align*}
\mathfrak H_n
:= 
\frac{\mathfrak V}{t_n^2 n} 
\sum_{k \in \mathbb Z} 
x_k \left( \frac{ x_k w_k}{|c_k|^2} - 1 \right)
 I\{ \lambda_k \geq \alpha \gamma_k^\nu \}
\sum_{i=1}^{n-1}  
\e|D_{i,k}|^2
\end{align*} and show
\begin{align*}
\mathfrak V_{n,1} = \mathfrak H_n + o(1)
\end{align*}
by proving the corresponding $L_2$-convergence. Afterwards we show that $\mathfrak H_n$ converges in probability to $\mathfrak V$. 
Writing for $i \in \{1,\ldots,n\}$ and $k \in \mathbb Z$ 
\begin{align*}
&|\mathscr U_{i,k} D_{i,k}|^2
\e |\mathscr U_{1,k}|^2 
-
\mathfrak V
\e|D_{i,k}|^2 \\
&=
\mathfrak V^{1/2}
\Big[  
|\mathscr U_{i,k} D_{i,k}|^2 
-
\mathfrak V^{1/2}
\e|D_{i,k}|^2
\Big] 
-
|\mathscr U_{i,k} D_{i,k}|^2 |\langle \beta, \phi_k \rangle|^2 x_k.
\end{align*}
and, observing that $\sigma^2 + \sum_{m \in \mathbb Z} |\langle \beta, \phi_m \rangle|^2 x_m  \leq C_1$ for some constant $C_1>0$, we get
\begin{align*}
&\e \left( \mathfrak V_{n,1} - \mathfrak H_n \right)^2 \leq \mathbb V_{n,1} + \mathbb V_{n,2} + \mathbb V_{n,3}
\end{align*}
%%%%%%%%%%%%%%%%%%%%%%%%%%%%%%%%%%%%%%%%%%%%%%%%%%%%%%%%%%%%%%%%%%%%%%%%%%%%%%%%%%%%%
with
\begin{align*}
\mathbb V_{n,1}&=
\frac{C}{t_n^4 n^2}
\sum_{k \in \mathbb Z} 
x_k^2 \left( \frac{ x_k w_k}{|c_k|^2} - 1 \right)^2
 I\{ \lambda_k \geq \alpha \gamma_k^\nu \} \\
&\phantom{=}
\phantom{\frac{C}{t_n^4 n^2} \sum_{k \in \mathbb Z}}
\Bigg\{
\sum_{i=1}^{n-1}
\e \Big(  
|\mathscr U_{i,k} D_{i,k}|^2 
 -  
\mathfrak V^{1/2}
\e|D_{i,k}|^2 \Big)^2\\
&\phantom{=}
\phantom{\frac{C}{t_n^4 n^2} \sum_{k \in \mathbb Z}\Big\{}
+
\sum_{\substack{i,p=1, \\ i \neq p}}^{n-1}
\e \Big[  
|\mathscr U_{i,k} D_{i,k}|^2 -  
\mathfrak V^{1/2}
\e|D_{1,k}|^2 \Big]
\e \Big[ (|\mathscr U_{p,k} D_{p,k}|^2 -  
\mathfrak V^{1/2}
\e|D_{1,k}|^2 \Big] 
\Bigg\},\\
%%%%%%%%%%%%%%%%%%%%%%%%%%%%%%%%%%%%%%%%%%%%%%%%%%%%%%%%%%%%%%%%%%%%%%%%%%%%%%%%%%%%%
\mathbb V_{n,2}&{=}
%+
\frac{C}{t_n^4 n^2}
\sum_{\substack{k,l \in \mathbb Z, \\ |k| \neq |l|}} 
x_k \left( \frac{ x_k w_k}{|c_k|^2} - 1 \right)
 I\{ \lambda_k \geq \alpha \gamma_k^\nu \} 
x_l \left( \frac{ x_l w_l}{|c_l|^2} - 1 \right)
 I\{ \lambda_l \geq \alpha \gamma_l^\nu \}  \\
&\phantom{=}
\phantom{+\frac{1}{t_n^4 n^2}} 
\phantom{\sum_{\substack{k,l \in \mathbb Z, \\ k \neq l}} }
\Bigg\{
\sum_{i=1}^{n-1}
\e \Big[ 
\Big( |\mathscr U_{i,k} D_{i,k}|^2  
-  
\mathfrak V^{1/2}
\e|D_{i,k}|^2\Big)
\Big( |\mathscr U_{i,l} D_{i,l}|^2 
-  
\mathfrak V^{1/2}
\e|D_{i,l}|^2\Big)
\Big] \\
&\phantom{=}
\phantom{+\frac{1}{t_n^4 n^2}} 
\phantom{\sum_{\substack{k,l \in \mathbb Z, \\ k \neq l}} }
+\sum_{\substack{i,p=1, \\ i \neq p}}^{n-1}
\e \Big[ 
|\mathscr U_{i,k} D_{i,k}|^2  
-  
\mathfrak V^{1/2}
\e|D_{i,k}|^2\Big] 
\e \Big[ |\mathscr U_{i,l} D_{i,l}|^2 
-  
\mathfrak V^{1/2}
\e|D_{i,l}|^2\Big] 
\Bigg\}, \\
%%%%%%%%%%%%%%%%%%%%%%%%%%%%%%%%%%%%%%%%%%%%%%%%%%%%%%%%%%%%%%%%%%%%%%%%%%%%%%%%%%%%%%%
\mathbb V_{n,3}&{=}
%+
\frac{2}{t_n^4 n^2}
\e \Bigg(  
\sum_{k \in \mathbb Z} 
x_k^2 \left( \frac{ x_k w_k}{|c_k|^2} - 1 \right)
|\langle \beta, \phi_k \rangle|^2
I\{ \lambda_k \geq \alpha \gamma_k^\nu \}
\sum_{i=1}^{n-1}
|\mathscr U_{i,k} D_{i,k}|^2  
\Bigg)^2. \\
\end{align*}
We have
\begin{align}
\e |\mathscr U_{j,k} D_{j,k} |^2 
=
\Bigg( \sigma^2 + \sum_{\substack{m \in \mathbb Z,\\ |m| \neq |k|}} 
|\langle \beta, \phi_m \rangle|^2 x_m \Bigg) 
\left( \frac{w_k}{|c_k|^2} - \frac{1}{x_k} \right), \label{UDquadkorrekt}
\end{align}
because $|\mathscr U_{j,k}|^2$ and $|D_{j,k}|^2$ are uncorrelated for all $k \in \mathbb Z$ and $j \in \{1,\ldots,n\}$. With Lemma
B.1 and \eqref{Dik4}, for all $i \in \{1,\ldots,n\}$ and $k \in
\mathcal K_n$, we have
\begin{align}
\e \left( |\mathscr U_{i,k} D_{i,k}|^2 
 - 
\mathfrak V^{1/2}
\e|D_{i,k}|^2 \right)^2 
&\leq
C \left( \e |D_{1,k}|^4 - \left( \e |D_{1,k}|^2 \right)^2 \right) \nn \leq
C \e |D_{1,k}|^4 \\
&\leq
\frac{C}{\alpha^2} \label{UDquaddiff}
\end{align}
as well as
\begin{align}
\e \left[ |\mathscr U_{i,k} D_{i,k}|^2
 -
\mathfrak V^{1/2}
\e|D_{i,k}|^2 \right] 
&=
- \e|D_{1,k}|^2 |\langle \beta, \phi_k \rangle|^2 x_k \nn \\
=
-\left( \frac{x_k w_k}{|c_k|^2} - 1 \right)|\langle \beta, \phi_k \rangle|^2.
\label{UDdiffvereinfacht}
\end{align}
For the mixed terms with $k,l \in \mathbb Z, |k| \neq |l|$ and $i \in \{1,\ldots,n\}$ and $\frac{w_k}{|c_k|^2} - \frac{1}{x_k} \geq 0$ for
all $k \in \mathbb Z$, we get
\begin{align}
&\e \Big[ 
\Big( |\mathscr U_{i,k} D_{i,k}|^2  
-  
\mathfrak V^{1/2}
\e|D_{i,k}|^2\Big) 
\Big( |\mathscr U_{i,l} D_{i,l}|^2 
-  
\mathfrak V^{1/2}
\e|D_{i,l}|^2\Big)
\Big] \nn \\
&\leq
\e \Big[  
| \mathscr U_{1,k} D_{1,k} \mathscr U_{1,l} D_{1,l}|^2
\Big]
+
\Big( \frac{w_k}{|c_k|^2} - \frac{1}{x_k} \Big)
\Big( \frac{w_l}{|c_l|^2} - \frac{1}{x_l} \Big) \nn \\
&\leq
C\Bigg\{\frac{1}{\alpha^2}
| \langle \beta, \phi_k \rangle |^2 x_k | \langle \beta, \phi_l \rangle |^2 x_l
+
\frac{x_l}{\alpha}
| \langle \beta, \phi_l \rangle |^2 \Big( \frac{w_k}{|c_k|^2} - \frac{1}{x_k} \Big) \nn \\
&\phantom{=}
\phantom{C\Bigg\{}
+
\frac{x_k}{\alpha}
| \langle \beta, \phi_k \rangle |^2 \Big( \frac{w_l}{|c_l|^2} - \frac{1}{x_l} \Big)
+
\Big( \frac{w_k}{|c_k|^2} - \frac{1}{x_k} \Big)
\Big( \frac{w_l}{|c_l|^2} - \frac{1}{x_l} \Big) \Bigg\}.  \label{UDgemdiff}
\end{align}
Using this, we have 
\begin{align*}
\mathbb V_{n,1}
&\leq
\frac{C}{t_n^4 n^2}
\sum_{k \in \mathbb Z} 
x_k^2 \left( \frac{ x_k w_k}{|c_k|^2} - 1 \right)^2
 I\{ \lambda_k \geq \alpha \gamma_k^\nu \} 
%\\
%&\phantom{=}
%\phantom{\frac{1}{t_n^4 n^2} \sum_{k \in \mathbb Z} }
\Bigg\{
\frac{n}{\alpha^2}
+
n^2 \left( \frac{x_k w_k}{|c_k|^2} - 1 \right)^2
|\langle \beta, \phi_k \rangle|^4
\Bigg\} \\
&\leq
\frac{C}{t_n^4 n \alpha^2}
\sum_{k \in \mathbb Z} 
x_k^2 \left( \frac{ x_k w_k}{|c_k|^2} - 1 \right)^2
I\{ \lambda_k \geq \alpha \gamma_k^\nu \} 
\\
&\phantom{=}
+
\frac{C}{t_n^4}
\sum_{k \in \mathbb Z} 
x_k^2 \left( \frac{ x_k w_k}{|c_k|^2} - 1 \right)^4
|\langle \beta, \phi_k \rangle|^4
 I\{ \lambda_k \geq \alpha \gamma_k^\nu \}	\\
&= o \left( 1 + \frac{1}{t_n^2} \right),
\end{align*}
with some constant $C>0$. With similar arguments, we obtain
\begin{align*}
\mathbb V_{n,2} 
&\leq
\frac{C}{t_n^4} \Bigg\{
\frac{1}{n \alpha^2}
\left(\sum_{k \in \mathbb Z} x_k^2 \left( \frac{x_k w_k}{|c_k|^2} - 1 \right)
|\langle \beta, \phi_k \rangle|^2
 I\{ \lambda_k \geq \alpha \gamma_k^\nu \} 
\right)^2 
\\
&\phantom{=}
\phantom{+\frac{C}{t_n^4} \Bigg\{}
+
\frac{t_n^2}{n \alpha}
\left(\sum_{l \in \mathbb Z} x_l^2 \left( \frac{x_l w_l}{|c_l|^2} - 1 \right)
|\langle \beta, \phi_l \rangle|^2
 I\{ \lambda_l \geq \alpha \gamma_l^\nu \} 
\right)
+
\frac{(t_n^2)^2}{n}  \Bigg\}\\
&\phantom{=}
+
\frac{C}{t_n^4}
\left(
\sum_{k \in \mathbb Z} 
x_k \left( \frac{ x_k w_k}{|c_k|^2} - 1 \right)^2
|\langle \beta, \phi_k \rangle|^2
 I\{ \lambda_k \geq \alpha \gamma_k^\nu \} 
\right)^2,
\end{align*}
which can be further bounded using the Cauchy-Schwarz inequality to get
\begin{align*}
\mathbb V_{n,2}
&=
o\left( 
1
+
\frac{1}{t_n^2}
+
\frac{1}{\sqrt{n} t_n}
\right)
+ \mathcal O \left( \frac{1}{n} \right).
\end{align*}
Using similar arguments as for the first two terms, $V_{n,3}$ can also be bounded to get
\begin{align*}
\mathbb V_{n,3} 
%%%%%%%%%%%%%%%%%%%%%%%%%%%%%%%%%%%%%%%%%%%%%%%%%%%%%%%%%%%%%%%%%%%%%%%%%%%%%%%%%%%%%%%%
%%%%%%%%%%%%%%%%%%%%%%%%%%%%%%%%%%%%%%%%%%%%%%%%%%%%%%%%%%%%%%%%%%%%%%%%%%%%%%%%%%%%%%%%
&\leq
\frac{C}{t_n^4 n \alpha^2}
\sum_{k \in \mathbb Z} 
x_k^4 \left( \frac{ x_k w_k}{|c_k|^2} - 1 \right)^2
|\langle \beta, \phi_k \rangle|^4
I\{ \lambda_k \geq \alpha \gamma_k^\nu \} \\
&\phantom{=}
+
\frac{C}{t_n^4}
\sum_{k \in \mathbb Z} 
x_k^2 \left( \frac{ x_k w_k}{|c_k|^2} - 1 \right)^4
|\langle \beta, \phi_k \rangle|^4
I\{ \lambda_k \geq \alpha \gamma_k^\nu \} \\
&\phantom{=}
+\frac{C}{t_n^4 n \alpha^2}  
\left(\sum_{k \in \mathbb Z} 
x_k^3 \left( \frac{ x_k w_k}{|c_k|^2} - 1 \right)
|\langle \beta, \phi_k \rangle|^4
I\{ \lambda_k \geq \alpha \gamma_k^\nu \}
\right)^2 \\
&\phantom{=}
+\frac{C}{t_n^4 n}  
\left(\sum_{k \in \mathbb Z} 
x_k \left( \frac{ x_k w_k}{|c_k|^2} - 1 \right)^2
|\langle \beta, \phi_k \rangle|^2
I\{ \lambda_k \geq \alpha \gamma_k^\nu \}
\right)^2 \\
&\phantom{=}
+
\frac{C}{t_n^4 n \alpha}
\sum_{k \in \mathbb Z} 
x_k \left( \frac{ x_k w_k}{|c_k|^2} - 1 \right)^2
|\langle \beta, \phi_k \rangle|^2
I\{ \lambda_k \geq \alpha \gamma_k^\nu \} \\
&\phantom{=}
\phantom{\frac{C}{t_n^4 n \alpha}\sum_{k \in \mathbb Z} }
\sum_{k \in \mathbb Z}
x_l^3 \left( \frac{ x_l w_l}{|c_l|^2} - 1 \right)
|\langle \beta, \phi_l \rangle|^4
I\{ \lambda_l \geq \alpha \gamma_l^\nu \}  \\
&\phantom{=}
+
\frac{C}{t_n^4 }
\left(
\sum_{k \in \mathbb Z} 
x_k \left( \frac{ x_k w_k}{|c_k|^2} - 1 \right)^2
|\langle \beta, \phi_k \rangle|^2
I\{ \lambda_k \geq \alpha \gamma_k^\nu \}
\right)^2 \\
&=
o \left( 1 + \frac{1}{t_n^2} + \frac1n + \frac{1}{\sqrt{n} t_n} \right).
\end{align*}
Altogether, we have
\[ \mathfrak V_{n,1} = \mathfrak H_n + o_{P} \left( 1 \right). \]
The stochastic convergence of $\mathfrak H_n$ follows by
\begin{align*}
\mathfrak H_n 
&=
\mathfrak V
\frac{n-1}{t_n^2 n} 
\sum_{k \in \mathbb Z} 
\left( \frac{ x_k w_k}{|c_k|^2} - 1 \right)^2
I\{ \lambda_k \geq \alpha \gamma_k^\nu \} 
\stackrel{P}{\to}\mathfrak V
\end{align*}
for $n\to\infty$.
%
%
%
%
%
%
%
%
%
%
%     ~ Diskusssion von Vn2  ~
%
%
%
%
%
%
%
%
%
%
For proving, that $\mathfrak V_{n,2}$ converges stochastically to 0 we show again the corresponding $L_2$-convergence. To this end we bound for all $i \in \{1,\ldots,n\}$ und
$k \in \mathbb Z$ the term $\e |\mathscr
U_{1,k}|^2$ by a constant $C<\infty$ using the centredness of $U$ and Lemma
B.1, to obtain
\begin{align*}
\e |\mathfrak V_{n,2}|^2 
%%%%%%%%%%%%%%%%%%%%%%%%%%%%%%%%%%%%%%%%%%%%%%%%%%%%%%%%%%%%%%%%%%%%%%%%%%%%%%%%%%%%%%%%%
&\leq
\frac{C}{t_n^4 n^2} \Bigg\{
\sum_{k \in \mathbb Z}
x_k^2 \left( \frac{ x_k w_k}{|c_k|^2} - 1 \right)^2 I\{ \lambda_k \geq \alpha \gamma_k^\nu
\} 
\e \Bigg|\sum_{\substack{i,p=1,\\ i \neq p}}^{n-1}
\mathscr U_{i,k} D_{i,k}
\ov{\mathscr U_{p,k} D_{p,k}} \Bigg|^2\\
&\phantom{=}
\phantom{\frac{C}{t_n^4 n^2} \Bigg\{}
+
\sum_{\substack{k,l \in \mathbb Z, \\ |k| \neq |l|}}
x_k \left( \frac{ x_k w_k}{|c_k|^2} - 1 \right) 
I\{ \lambda_k \geq \alpha \gamma_k^\nu \} 
x_l \left( \frac{ x_l w_l}{|c_l|^2} - 1 \right) 
I\{ \lambda_l \geq \alpha \gamma_l^\nu \}  \\
&\phantom{=}
\phantom{\frac{C}{t_n^4 n^2} \Bigg\{ + \sum_{\substack{k,l \in \mathbb Z, \\ k \neq l}}}
\e \Bigg[
\Bigg(\sum_{\substack{i,p=1,\\ i \neq p}}^{n-1}
\mathscr U_{i,k} D_{i,k}
\ov{\mathscr U_{p,k} D_{p,k}}\Bigg) 
\Bigg(\sum_{\substack{i,p=1,\\ i \neq p}}^{n-1}
\ov{\mathscr U_{i,l} D_{i,l}}
\mathscr U_{p,l} D_{p,l} \Bigg)
\Bigg] \Bigg\}.
\end{align*}
Since $\mathscr U_{i,k} D_{i,k}$ and $\mathscr U_{p,k} D_{p,k}$ are stochastically independent for $p \neq i$, only the quadratic terms for $k \in \mathbb Z$ are relevant 
\begin{align*}
\sum_{\substack{i,p=1,\\ i \neq p}}^{n-1}
\e \Big|
\mathscr U_{i,k} D_{i,k}
\ov{\mathscr U_{p,k} D_{p,k}} \Big|^2
&=
\sum_{\substack{i,p=1,\\ i \neq p}}^{n-1}
\e |\mathscr U_{i,k} D_{i,k}|^2
\e |\mathscr U_{p,k} D_{p,k} |^2 \\
&=
(n-1)(n-2) \left(\e |\mathscr U_{1,k}|^2 \e |D_{1,k}|^2 \right)^2 \\
&\leq
C n^2 \left( \frac{w_k}{|c_k|^2} - \frac{1}{x_k} \right)^2.
\end{align*}
Under the assumptions of Theorem 2.1, this leads to
\begin{align*}
\e |\mathfrak V_{n,2}|^2
&\leq
\frac{C}{t_n^4}
\sum_{k \in \mathbb Z}
\left( \frac{ x_k w_k}{|c_k|^2} - 1 \right)^4 
I\{ \lambda_k \geq \alpha \gamma_k^\nu \} = o(1),
\end{align*}
and therefore
\[ \mathfrak V_{n,2} = o_P \left( 1 \right). \]

\subsection{Details for the proof of Proposition C.2}

It is shown in [1] and [10] that the conditional Lindeberg condition
follows from the unconditional Ljapunov condition. We will show in the following, that \[ \sum_{j=2}^n \e|Y_{n,j}|^4 = o(1) \]
and decompose
\[\sum_{j=2}^n \e|Y_{n,j}|^4=L_{n,1}+L_{n,2}+L_{n,3}+L_{n,4},\]
where
\begin{align*}
L_{n,1}
&=
\frac{1}{t_n^4 n^4} \sum_{j=2}^n 
\sum_{k \in \mathbb Z} 
\e \left|\mathscr U_{j,k} D_{j,k} Z_{n,j,k}\right|^4,\\
L_{n,2}&=
\frac{1}{t_n^4 n^4} \sum_{j=2}^n\sum_{\substack{k,l \in \mathbb Z,\\ |k| \neq |l|}} 
\e \left|\mathscr U_{j,k} D_{j,k} Z_{n,j,k} \ov{\mathscr U_{j,l} D_{j,l}
Z_{n,j,l}}\right|^2, \\
L_{n,3}&{=}
{\frac{1}{t_n^4 n^4} \sum_{j=2}^n }
\sum_{\substack{k,l,q \in \mathbb Z,\\ |k|,|l| \neq |q|, |k| \neq |l|}} 
\e \Big[|\mathscr U_{j,k} D_{j,k} Z_{n,j,k}|^2 
\mathscr U_{j,l} D_{j,l} Z_{n,j,l}
\ov{\mathscr U_{j,q} D_{j,q} Z_{n,j,q}}\Big], \\
L_{n,4}&{=}
{\frac{1}{t_n^4 n^4} \sum_{j=2}^n }
\sum_{\substack{k,l,p,q \in \mathbb Z,\\ |k|,|l|,|p| \neq |q|, \\ |k|,|l| \neq |p|, |k|
\neq |l|}} \e \Big[\mathscr U_{j,k} D_{j,k} Z_{n,j,k} 
\ov{\mathscr U_{j,l} D_{j,l} Z_{n,j,l}}
\mathscr U_{j,p} D_{j,p} Z_{n,j,p}
\ov{\mathscr U_{j,q} D_{j,q} Z_{n,j,q}}\Big].
\end{align*}
%
%
%
%
%
%
% ------- Ln1 ---------
%
%
%
%
%
%
%
For $L_{n,1}$ we use, that for all $k \in \mathbb Z, n \in
\mathbb N, j \in \{1,\ldots,n\}$, $Z_{n,j,k}$ are stochastically independent of $\mathscr
U_{j,k} D_{j,k}$ and $\mathscr U_{j,k}$ are uncorrelated with $D_{j,k}$. Furthermore, the
fourth absolute moment of $\mathscr U_{j,k}$ is due to the centredness of $U$ and Lemma
B.1 uniformly bounded. The fourth absolute moment of $D_{j,k}$ can be
estimated using Assumption 3 and $(X,W) \in \mathcal F_{\eta}^4$ as
\begin{align}
\e |D_{j,k}|^4 
\leq 
C \left( \frac{\e |\langle W, \phi_k \rangle |^4}{|c_k|^4} 
+
\frac{\e |\langle X, \phi_k \rangle |^4}{x_k^4} \right)
\leq
C \eta \left( \frac{w_k^2}{|c_k|^4} 
+
\frac{1}{x_k^2} \right)
\leq
\frac{C \eta}{\alpha^2}. \label{Dik4}
\end{align} 
Again using similar arguments, we obtain
\begin{align}
\e \left| 
\mathscr U_{i_1,k} D_{i_1,k}\right|^2 
=
\e | \mathscr U_{i_1,k} |^2 \e |D_{i_1,k}|^2 
&\leq
C \left( \frac{w_k}{|c_k|^2} - \frac{1}{x_k} \right). \label{UDquad}
\end{align}
This results in
\begin{align}
&\e \Big|
\sum_{i=1}^{j-1}
\mathscr U_{i,k}
D_{i,k}
x_k I\{ \lambda_k \geq \alpha \gamma_k^\nu \} 
\Big|^4 \nn \\
%%%%%%%%%%%%%%%%%%%%%%%%%%%%%%%%%%%%%%%%%%%%%%%%%%%%%%%%%%%%%%%%%%%%%%%%%%%%%%%%%%%%%
&=
x_k^4 I\{ \lambda_k \geq \alpha \gamma_k^\nu \} 
\Bigg\{
\sum_{i=1}^{j-1}
\e |\mathscr U_{i,k}|^4
\e |D_{i,k}|^4
+
2 \sum_{1 \leq i_1 < i_2 \leq j-1}
\e |\mathscr U_{i_1,k} D_{i_1,k}|^2
\e |\mathscr U_{i_2,k} D_{i_2,k}|^2
\Bigg\} \nn \\
&\leq
\frac{Cn}{\alpha^2} x_k^4 I\{ \lambda_k \geq \alpha \gamma_k^\nu \} 
+
C n^2 x_k^2 
\left( \frac{x_k w_k}{|c_k|^2} - 1 \right)^2
I\{ \lambda_k \geq \alpha \gamma_k^\nu \}. \label{Znjk4}
\end{align}
Putting these results together, for $L_{n,1}$, we get
\begin{align*}
L_{n,1}
&=
\frac{1}{t_n^4 n^4} \sum_{j=2}^n 
\sum_{k \in \mathbb Z} 
\e |\mathscr U_{j,k}|^4 
\e |D_{j,k}|^4 
\e |Z_{n,j,k}|^4 \\
&\leq
\frac{C}{t_n^4 n^4 \alpha^2} \sum_{j=2}^n 
\sum_{k \in \mathbb Z} 
\e \Big|
\sum_{i=1}^{j-1}
\mathscr U_{i,k}
D_{i,k}
x_k I\{ \lambda_k \geq \alpha \gamma_k^\nu \} 
\Big|^4 \\
&\leq
\frac{C}{t_n^4 n \alpha^2}
\sum_{k \in \mathbb Z} 
x_k^2
I\{ \lambda_k \geq \alpha \gamma_k^\nu \} 
\left(
\frac{1}{n \alpha^2} x_k^2  
+
\left( \frac{x_k w_k}{|c_k|^2} - 1 \right)^2  
\right) \\
&=
o(1) \frac{1}{t_n^4} 
\left( \sum_{k \in \mathbb Z} 
 x_k^4
I\{ \lambda_k \geq \alpha \gamma_k^\nu \} 
+
\sum_{k \in \mathbb Z} 
 x_k^2
\left( \frac{x_k w_k}{|c_k|^2} - 1 \right)^2
I\{ \lambda_k \geq \alpha \gamma_k^\nu \}\right),
\end{align*}
where the first series converges due to Lemma B.1 and the second series either also converges or, if not, can be bounded by $C t_n^2$. 
%
%
%
%
%
%
% ------- Ln4 ---------
%
%
%
%
%
%
% 
\bigskip

Considering $L_{n,4}$, we use the stochastic independence of $Z_{n,j,k}$ and $\mathscr U_{j,l}
D_{j,l}$ for all $k,l \in \mathbb Z$, which results in
\begin{align*}
&\e \big[\mathscr U_{j,k} D_{j,k} Z_{n,j,k} 
\ov{\mathscr U_{j,l}} \ov{D_{j,l}} \ov{Z_{n,j,l}}
\mathscr U_{j,p} D_{j,p} Z_{n,j,p}
\ov{\mathscr U_{j,q}} \ov{D_{j,q}} \ov{Z_{n,j,q}}\big] \\
&=
\e \big[\mathscr U_{j,k} D_{j,k}  
\ov{\mathscr U_{j,l}} \ov{D_{j,l}} 
\mathscr U_{j,p} D_{j,p} 
\ov{\mathscr U_{j,q}} \ov{D_{j,q}} \big]
\e \big[ Z_{n,j,k} \ov{Z_{n,j,l}} Z_{n,j,p} \ov{Z_{n,j,q}}\big].
\end{align*}
The rest of the argumentation is just calculating the expectations using that
for all $j \in \{1,\ldots,n\}$, 
$D_{j,k}, D_{j,l}, D_{j,p}$ and $D_{j,q}$ are uncorrelated with $S_{j,m}$ for all $m \in
\mathbb Z \backslash \{m \in \mathbb Z: |m| = |k|,|l|,|p|,|q|\}$ and stochastically independent of $U_j$. Finally,
\begin{align}
\e [S_{j,k} D_{j,k}]
&=
\langle \beta, \phi_k \rangle
\e \left[
\langle \phi_k, X_j \rangle 
\left( 
\frac{\langle W_j, \phi_k \rangle}{c_k} - \frac{\langle X_j, \phi_k \rangle}{x_k}
\right)\right] =
\langle \beta, \phi_k \rangle
\left( \frac{c_k}{c_k} - \frac{x_k}{x_k} \right)=0 \label{SD}
\end{align}
and, in the same way, $\e [\ov{S_{j,k}}D_{j,k}]=\e [S_{j,k} \ov{D_{j,k}}]=0$, which gives $L_{n,4}=0$.
%
%
%
%
%
%
% ------- Ln2 ---------
%
%
%
%
%
%
% 
\bigskip

With similar arguments as above, we get
\begin{align*}
L_{n,2}
&=
\frac{1}{t_n^4 n^4} \sum_{j=2}^n \sum_{\substack{k,l \in \mathbb Z, \\ k \neq l}}
\e |\mathscr U_{j,k} D_{j,k} \ov{\mathscr U_{j,l}} \ov{D_{j,l}}|^2
\e | Z_{n,j,k} \ov{Z_{n,j,l}} |^2,
\end{align*}
which can be further bounded by using
\begin{align*}
\e \big|\ov{S_{j,k}} D_{j,k}\big|^2
&\leq
| \langle \beta, \phi_k \rangle |^2
\sqrt{\e \big|\langle X, \phi_k \rangle|^4 \e \big|D_{j,k}\big|^4 } \nn \\
&\leq
\sqrt{\eta}| \langle \beta, \phi_k \rangle |^2 x_k 
\left( \frac{\e |\langle W, \phi_k \rangle |^4}{|c_k|^4} +
\frac{\e |\langle X, \phi_k \rangle |^4}{x_k^4} \right)^{1/2} \nn \\
&\leq
C | \langle \beta, \phi_k \rangle |^2 x_k \left( \frac{w_k^2}{|c_k|^4} 
+
\frac{1}{x_k^2} \right)^{1/2}
\leq
\frac{C | \langle \beta, \phi_k \rangle |^2 x_k}{\alpha}
\end{align*}
and
\begin{align*}
&\e | Z_{n,j,k} \ov{Z_{n,j,l}} |^2 \\
&\leq
C x_k^2 x_l^2 
I\{ \lambda_k \geq \alpha \gamma_k^\nu \}
I\{ \lambda_l \geq \alpha \gamma_l^\nu \} (n-1)\\
&\phantom{=}
\Bigg\{
\Bigg[ 
\frac{C}{\alpha^2} | \langle \beta, \phi_k \rangle |^2 x_k |\langle \beta, \phi_l \rangle|^2 x_l 
+
\frac{C | \langle \beta, \phi_l \rangle |^2 x_l}{\alpha} 
\left( \frac{w_k}{|c_k|^2} - \frac{1}{x_k} \right)  \nn \\
&\phantom{=}
\phantom{C \big\{}
+
\frac{C | \langle \beta, \phi_k \rangle |^2 x_k}{\alpha} 
\left( \frac{w_l}{|c_l|^2} - \frac{1}{x_l} \right)
+
\left( \frac{w_k}{|c_k|^2} - \frac{1}{x_k} \right)
\left( \frac{w_l}{|c_l|^2} - \frac{1}{x_l} \right)
\Bigg]\\
&\phantom{=}
\phantom{\Big\{}
+
(n-2)
\left( \frac{w_k}{|c_k|^2} - \frac{1}{x_k} \right)
\left( \frac{w_l}{|c_l|^2} - \frac{1}{x_l} \right)
\Bigg\}.
\end{align*}
This results in
\begin{align*}
L_{n,2} 
&\leq
\frac{C}{t_n^4 (n \alpha^2)^2} \left(\sum_{k \in \mathbb Z} | \langle \beta, \phi_k
\rangle |^4 x_k^4\right)^2 
+
\frac{C}{t_n^2 n^2 \alpha^2} 
\sum_{l \in \mathbb Z} |\langle \beta, \phi_l \rangle |^4 x_l^4
+
\frac{C}{n^2} \\
&\phantom{=}
+
\frac{C}{t_n^4 n \alpha^2} \left(\sum_{k \in \mathbb Z}
|\langle \beta, \phi_k \rangle |^2 x_k^2 
\left( \frac{w_k x_k}{|c_k|^2} - 1 \right)
I\{ \lambda_k \geq \alpha \gamma_k^\nu \}
\right)^2 \\
&\phantom{=}
+
\frac{C}{t_n^2 n \alpha} \sum_{l \in \mathbb Z}
|\langle \beta, \phi_l \rangle |^2 x_l^2 
\left( \frac{w_l x_l}{|c_l|^2} - 1 \right)
I\{ \lambda_l \geq \alpha \gamma_l^\nu \}
+
\frac{C}{n} \\
&\leq
o\left(\frac{1}{t_n^4} + \frac{1}{t_n^2 n} 
\right)
+
\mathcal O \left(\frac1n +\frac{1}{n^2} \right) 
+
\frac{C}{t_n^2 n \alpha^2} 
\sum_{k \in \mathbb Z}
|\langle \beta, \phi_k \rangle |^4 x_k^4 
+
\frac{C}{t_n n \alpha} 
\sqrt{\sum_{k \in \mathbb Z}
|\langle \beta, \phi_k \rangle |^4 x_k^4 }
\\
&=
o\left(\frac{1}{t_n^4} + \frac{1}{t_n^2 n} 
+ \frac{1}{t_n^2}
+  \frac{1}{t_n \sqrt{n}}
\right)
+
\mathcal O \left(\frac1n +\frac{1}{n^2} \right) \\
&=
o(1),
\end{align*}
using the H\"older inequality and Lemma B.1.
%
%
%
%
%
%
%
% ------- Ln3 ---------
%
%
%
%
%
%
% 
\bigskip

For the summands in $L_{n,3}$, we get
\begin{align*}
&\e \big[|\mathscr U_{j,k} D_{j,k} Z_{n,j,k}|^2 
\mathscr U_{j,l} D_{j,l} Z_{n,j,l}
\ov{\mathscr U_{j,q} D_{j,q} Z_{n,j,q}}\big] \\
&=
\e \big[|\mathscr U_{j,k} D_{j,k} |^2 
\mathscr U_{j,l} D_{j,l} 
\ov{\mathscr U_{j,q} D_{j,q}} \big]
\e \big[ |Z_{n,j,k}|^2 Z_{n,j,l} \ov{Z_{n,j,q}}\big].
\end{align*}
The first expectation is
\begin{align*}
&\e \big[|\mathscr U_{j,k} D_{j,k} |^2 
\mathscr U_{j,l} D_{j,l} 
\ov{\mathscr U_{j,q} D_{j,q}} \big]  \\
&=
\left( \frac{w_k}{|c_k|^2} - \frac{1}{x_k} \right)
|\langle \beta, \phi_l \rangle |^2
|\langle \beta, \phi_q \rangle |^2  \\
&\phantom{=}\phantom{\Bigg(}
\e \left[ |\langle X_j, \phi_l \rangle |^2 
\left( \frac{\langle W_j, \phi_l \rangle}{c_l} - \frac{\langle X_j, \phi_l \rangle}{x_l}
\right) \right]
\e \left[ |\langle X_j, \phi_q \rangle |^2 
\left( \frac{\langle \phi_q, W_j \rangle}{\ov{c_q}} - \frac{\langle \phi_q, X_j
\rangle}{x_q} \right) \right], 
\end{align*}
while
\begin{align*}
&\e \big[ |Z_{n,j,k}|^2 Z_{n,j,l} \ov{Z_{n,j,q}}\big] \\
&=
x_k^2 x_l x_q
I\{ \lambda_k \geq \alpha \gamma_k^\nu \}
I\{ \lambda_l \geq \alpha \gamma_l^\nu \}
I\{ \lambda_q \geq \alpha \gamma_q^\nu \}
\sum_{i=1}^{j-1}
\e \big[|\mathscr U_{i,k} D_{i,k} |^2 
\mathscr U_{i,l} D_{i,l} 
\ov{\mathscr U_{i,q}} \ov{D_{i,q}} \big].
\end{align*}
Altogether, we have
\begin{align*}
&L_{n,3} \leq
\frac{1}{t_n^4 n^2}
\sum_{\substack{ k,l,q \in \mathbb Z, \\ |k|,|l| \neq |q|, |k| \neq |l|}} 
x_k^2 x_l x_q
I\{ \lambda_k \geq \alpha \gamma_k^\nu \}
I\{ \lambda_l \geq \alpha \gamma_l^\nu \}
I\{ \lambda_q \geq \alpha \gamma_q^\nu \} \\
&\phantom{=}
\phantom{\frac{1}{t_n^4 n^4}}
\left( \frac{w_k}{|c_k|^2} - \frac{1}{x_k} \right)^2
|\langle \beta, \phi_l \rangle |^4
|\langle \beta, \phi_q \rangle |^4  \\
&\phantom{=}
\phantom{\frac{1}{t_n^4 n^4}}
\left(\e \left[ |\langle X, \phi_l \rangle |^2 
\left( \frac{\langle W, \phi_l \rangle}{c_l} - \frac{\langle X, \phi_l \rangle}{x_l}
\right) \right]
\e \left[ |\langle X, \phi_q \rangle |^2 
\left( \frac{\langle \phi_q, W \rangle}{\ov{c_q}} - \frac{\langle \phi_q, X \rangle}{x_q}
\right) \right]\right)^2.
\end{align*}
The series can be bounded by $t_n^2$. Using the H\"older inequality for $l \in \mathcal K_n$, we have
\begin{align*}
\left(\e \Big[ |\langle X, \phi_l \rangle |^2 
\Big( \frac{\langle \phi_l, W \rangle}{\ov{c_l}} - \frac{\langle \phi_l, X \rangle}{x_l}
\Big) \Big]\right)^2
\leq
\e|\langle X, \phi_l \rangle |^4
\e \Big| D_{1,l}|^2\leq
\eta x_l^2
\left( \frac{w_l}{|c_l|^2} - \frac{1}{x_l} \right) 
\leq
\frac{C}{\alpha^2} x_l^2.
\end{align*} 
Finally, relying again on Assumption 3 and Lemma B.1, also
$L_{n,3}$ converges to 0 due to
\begin{align*}
L_{n,3}
\leq
\frac{C}{t_n^2 n^2}\left( \sum_{k \in \mathbb Z} 
|\langle \beta, \phi_k \rangle |^4 x_k \frac{x_k - \lambda_k}{\lambda_k}
\right)^{1/2} \leq
\frac{C}{t_n^2 n^2 \alpha^2}
\sum_{k \in \mathbb Z} 
|\langle \beta, \phi_k \rangle |^4 x_k (x_k - \lambda_k)
= o \left( \frac{1}{t_n^2 n} \right).
\end{align*}

\end{appendix}
%%%%%%%%%%%%%%%%%%%%%%%%%%%%%%%%%%%%%%%%%%%%%%
%% Multiple Appendixes:                     %%
%%%%%%%%%%%%%%%%%%%%%%%%%%%%%%%%%%%%%%%%%%%%%%
%\begin{appendix}
%\section{???}
%
%\section{???}
%
%\end{appendix}

%%%%%%%%%%%%%%%%%%%%%%%%%%%%%%%%%%%%%%%%%%%%%%
%% Support information, if any,             %%
%% should be provided in the                %%
%% Acknowledgements section.                %%
%%%%%%%%%%%%%%%%%%%%%%%%%%%%%%%%%%%%%%%%%%%%%%
\paragraph{Acknowledgments}
The authors would like to thank Jan Johannes for helpfull discussions on the different estimation techniques in the functional linear regression model with and without endogeneity.

\end{document}